\documentclass[12pt]{amsart}
\usepackage{amsmath, amssymb, amsthm}
\usepackage{mathrsfs}

\usepackage[refpage]{nomencl}
\makenomenclature

 \newtheorem{thm}{Theorem}[section]
\newtheorem*{theorem*}{Theorem}
 \newtheorem{cor}[thm]{Corollary}
 \newtheorem{lem}[thm]{Lemma}
 \newtheorem{prop}[thm]{Proposition}
 \theoremstyle{definition}
  \newtheorem{defn}[thm]{Definition}
 \newtheorem{rem}[thm]{Remark}
 \theoremstyle{definition}
 \newtheorem{ex}[thm]{Example}
 \numberwithin{equation}{section}
 \newtheorem{question}[thm]{Question}

\begin{document}

\author{Kelly Bickel}
\address{Georgia Institute of Technology, School of Mathematics,
  Atlanta, GA 30308}
\address{Bucknell University, Department of Mathematics, Lewisburg, PA 17837}
\email{kbickel3@math.gatech.edu}

\author{Greg Knese}
\address{Washington University in St. Louis, Department of
  Mathematics, St. Louis, MO 63130}
\email{geknese@math.wustl.edu}
\date{\today}

\thanks{GK supported by NSF grant DMS-1048775}

\keywords{Agler decomposition, scattering systems, transfer function,
  colligation, Lax-Phillips, input/state/output system, bidisk,
  polydisk, Schur class, de Branges-Rovnyak spaces, analytic
  extension, reproducing kernel Hilbert space, operator ranges}

\subjclass[2010]{Primary 47B32; Secondary 47A57, 47A40, 93C35, 46E22}

\title[Canonical Agler Decompositions and T.F.R.'s]{Canonical Agler Decompositions and Transfer Function Realizations}
\maketitle

%\onehalfspacing

\newcommand{ \LL}{\left \langle}
\newcommand{ \RR}{\right \rangle}
\newcommand{ \D}{\mathbb{D}}
\newcommand{\C}{\mathbb{C}}
\newcommand{\E}{\mathbb{E}}
\newcommand{\T}{\mathbb{T}}
\newcommand{\mcH}{\mathcal{H}}
\newcommand{\mcK}{\mathcal{K}}
\newcommand{ \K}{\mathcal{K}}
\newcommand{ \U}{\mathcal{U}}
\newcommand{ \F}{\mathcal{F}}
\newcommand{ \W}{\mathcal{W}}
\newcommand{\Q}{\mathcal{Q}}
\newcommand{\M}{\mathcal{M}}
\newcommand{\mcR}{\mathcal{R}}
\newcommand{ \WT}{\widetilde{\mathcal{W}}}
\newcommand{ \scrH}{\mathscr{H}}
\newcommand{ \ZZ}{\mathbb{Z}}
\newcommand{ \NN}{\mathbb{N}}
\newcommand{ \beq}{\begin{equation*}}
\newcommand{ \eeq}{\end{equation*}}
\newcommand{\supp}{\text{supp}}
\newcommand{ \Dphi}{D_{\Phi^*}}
\newcommand{ \Kphi}{\mathcal{K}_{\Phi}}
\newcommand{ \Hphi}{\mathcal{H}_{\Phi}}
\newcommand{ \SEE}{\mathcal{S}_2(E,E_*)}
\newcommand{\op}{\begin{bmatrix} I & \Phi \\ \Phi^* & I \end{bmatrix}}
\newcommand{\vecf}{\begin{bmatrix} f \\ g \end{bmatrix}}
\newcommand{\vecphi}{ \left[ \begin{array}{c}
\Phi \\  I \end{array} \right]}
\newcommand{\vecphis}{ \left[ \begin{array}{c}
I \\ \Phi^* \end{array} \right]}
\newcommand{\vecft}{ \left[ \begin{array}{c}
f' \\ g' \end{array} \right]}
\newcommand{\vecs}[2]{\begin{bmatrix} #1 \\ #2 \end{bmatrix}}

\newcommand{\Kmax}{ \mathcal{H}(K^{max}_j)}
\newcommand{\Kmin}{\mathcal{H}(K^{min}_j)}

\newcommand{\ip}[2]{\left\langle #1, #2 \right\rangle}

\bibliographystyle{plain}

\begin{abstract}
  A seminal result of Agler proves that the natural de Branges-Rovnyak
  kernel function associated to a bounded analytic function on the
  bidisk can be decomposed into two shift-invariant pieces. Agler's
  decomposition is non-constructive---a problem remedied by work of
  Ball-Sadosky-Vinnikov, which uses scattering systems to produce
  Agler decompositions through concrete Hilbert space geometry.  This
  method, while constructive, so far has not revealed the rich
  structure shown to be present for special classes of
  functions---inner and rational inner functions.  In this paper, we
  show that most of the important structure present in these
  special cases extends to general bounded analytic functions.  We
  give characterizations of all Agler decompositions, we prove the
  existence of coisometric transfer function realizations with natural
  state spaces, and we characterize when Schur functions on the bidisk
  possess analytic extensions past the boundary in terms of associated
  Hilbert spaces.
\end{abstract}

\section{Introduction}

Let $E$ and $E_*$ be separable Hilbert spaces and recall that the
\emph{Schur class $\mathcal{S}_d(E, E_*)$}
\nomenclature{$\mathcal{S}_d(E, E_*) $}{$d$ variable Schur Class} is
the set of holomorphic functions $\Phi: \D^d \rightarrow
\mathcal{L}(E,E_*)$ such that each $\Phi(z): E \rightarrow E_*$ is a
linear contraction. In one variable, the structure of these functions
is well-understood and they play key roles in many areas of both pure
and applied mathematics. For example, they are objects of interest
in $H^{\infty}$ control theory, act as scattering functions of
single-evolution Lax-Phillips scattering systems, and serve as the
transfer functions of one-dimensional dissipative, linear,
discrete-time input/state/output (i/s/o) systems \cite{bsv05, he74,
  hj99}.  Moreover, every $\Phi \in \mathcal{S}_1(E, E_*)$ can actually
be realized as both a scattering function of a Lax-Phillips scattering
system and a transfer function of a dissipative, linear, discrete-time
i/s/o system.  For simplicity, we omit the discussion of the connection
to the interesting topic of von Neumann inequalities; see \cite{ampi,
  bsv05, gkvw08}.

The situation in several variables is more complicated; although Schur
functions are still the scattering functions of $d$-evolution
scattering systems and transfer functions of $d$-dimensional dissipative, 
linear, discrete-time i/s/o systems, the converse is not
always true; there are functions in $\mathcal{S}_d(E,E_*)$ that cannot
be realized as transfer functions of dissipative i/s/o systems.
To make this precise, let $\mathcal{M}= \mathcal{M}_1 \oplus \dots
\oplus \mathcal{M}_d$ be a separable Hilbert space, and  for each $z \in \mathbb{D}^d$, 
define the multiplication operator $\mathcal{E}_z: = z_1P_{\mathcal{M}_1} + \dots + z_d
P_{\mathcal{M}_d},$ where each
$P_{\mathcal{M}_r}$ is the projection onto $\mathcal{M}_{r}$.

\begin{defn} Let $\Phi \in \mathcal{S}_d(E,E_*).$ A \emph{Transfer Function Realization}
 (T.F.R.) of $\Phi$ consists of a Hilbert space $\mathcal{M}= \mathcal{M}_1 \oplus \dots
\oplus \mathcal{M}_d$ and a contraction $U:  \mathcal{M} \oplus E  \rightarrow
\mathcal{M} \oplus E_*$ such that if $U$ is written as
\[
U = 
\left[ \begin{array}{cc} 
A & B \\
C & D 
\end{array} \right] : 
\left[  \begin{array}{c}
\mathcal{M} \\
 E 
\end{array}  \right]
\rightarrow 
\left[  \begin{array}{c}
\mathcal{M} \\
 E_*  \end{array} \right], 
\]
then $\Phi(z) = D + C\left( I_{\mathcal{M}} - \mathcal{E}_zA \right)^{-1}\mathcal{E}_z B$.
The Hilbert space $\mathcal{M}$ is called the \emph{state space} and
the contraction $U$ is called the \emph{colligation}.  One can
associate a d-dimensional dissipative, linear, discrete-time $i/s/o$ system with the
pair $(\mathcal{M}, U)$. The transfer function realization is called
isometric, coisometric, or unitary whenever $U$ is isometric,
coisometric, or unitary.\end{defn}

In \cite{ag1,ag90}, J. Agler showed that every function in
$\mathcal{S}_2(E,E_*)$ has a T.F.R. and used the realizations to
generalize the Pick interpolation theorem to two variables. Since
Agler's seminal results, these formulas have been used frequently to
both generalize one-variable results and address strictly multivariate
questions on the polydisc as in \cite{agmc_isb, agmc_dv, mcc10a,
  amy10a, baltre98, kn07b, kn08ua, mcc12}.  There is also a simple
relationship between transfer function realizations and positive
kernels:

\begin{thm} \label{thm1} \emph{(Agler \cite{ag90}).} Let $\Phi \in \mathcal{S}_d(E,E_*)$. Then,
$\Phi$ has a transfer function realization if and only if there 
are positive holomorphic kernels $K_1, \dots, K_d: \mathbb{D}^d \times \mathbb{D}^d \rightarrow \mathcal{L}(E_*)$ such that for all
$z,w \in \D^d$
\begin{equation*}  I_{E_*}  - \Phi(z) \Phi(w)^* = (1-z_1 \bar{w}_1) K_1(z,w) + \dots + (1-z_d
\bar{w}_d) K_d(z,w).  \end{equation*} 
\end{thm}
This decomposition using positive kernels is called an \emph{Agler
  decomposition} of $\Phi$.
%  and the kernels $(K_1, \dots, K_d)$ are
% called \emph{Agler kernels} of $\Phi$.
In two variables, it is convenient to reverse the ordering, and throughout
this paper, positive kernels $(K_1, K_2)$ are called \emph{Agler kernels of}
$\Phi \in \mathcal{S}_2(E,E_*)$ if for all $z,w \in \D^2$
\begin{equation}\label{eqn:agdecomp}  I_{E_*} - \Phi(z) \Phi(w)^* = (1 -z_1 \bar{w}_1) K_2(z,w) +
  (1-z_2\bar{w}_2) K_1(z,w).  \end{equation} \nomenclature{$(K_1,K_2)$}{Agler kernels of $\Phi$}

Agler proved the existence of a pair of Agler kernels for each
function in $\SEE$ and then showed this gives a transfer function
realization via Theorem \ref{thm1}.  It is often easier to go from
kernels to realizations because positive kernels immediately bring
operator theory and reproducing kernel Hilbert space methods into the
picture. We review some of these concepts related to positive kernels
below.

\begin{rem} 
Recall that $K : \Omega \times \Omega \rightarrow \mathcal{L}(E)$ is a
\emph{positive kernel on $\Omega$} if for each $N \in \NN$
\[ \sum_{i,j =1}^N \LL K(x_i,x_j) \eta_j, \eta_i \RR_{E} \ge 0 \]
for all $x_1, \dots, x_N \in \Omega$ and $\eta_1, \dots, \eta_N \in
E.$ Similarly, $\mathcal{H}$ is a \emph{reproducing kernel Hilbert
  space on $\Omega$} if $\mathcal{H}$ is a Hilbert space of functions on
defined  $\Omega$ such that evaluation at $x$ is a bounded linear
operator for each $x \in \Omega.$ Then there is a unique positive
kernel $K: \Omega \times \Omega \rightarrow \mathcal{L}(E)$ with
\[ 
\LL f, K(\cdot, y) \eta \RR_{\mathcal{H}} = \LL f(y), \eta
\RR_{E} \qquad \forall \ f \in \mathcal{H}, y \in \Omega, \text{ and }
\eta \in E. 
\]
Conversely, given any positive kernel $K$ on $\Omega$, there is a
reproducing kernel Hilbert space, denoted $\mathcal{H}(K)$, on $\Omega$ with $K$
as its reproducing kernel. For details, see \cite{bv03b}.  \end{rem}

The kernels $K_1,K_2$ are written in reverse order in
\eqref{eqn:agdecomp} because upon dividing the equation through by
$(1-z_1\bar{w}_1)(1-z_1\bar{w}_2)$, an Agler decomposition can be
given a much more natural interpretation in terms of de
Branges-Rovnyak spaces.

\begin{rem} Assume $(K_1, K_2)$ are Agler kernels of $\Phi$ and rewrite (\ref{eqn:agdecomp}) 
as follows:
\begin{equation} \label{eqn:agdecomp2} \frac{I -\Phi(z) \Phi(w)^*}{(1-z_1\bar{w}_1) (1-z_2\bar{w}_2)} 
= \frac{K_1(z,w)}{1-z_1\bar{w}_1} + \frac{K_2(z,w)}{1-z_2\bar{w}_2}. 
\end{equation}
Each term in (\ref{eqn:agdecomp2}) is a positive kernel and so, we can define the following 
reproducing kernel Hilbert spaces:
\[ \mathcal{H}_{\Phi}:= \mathcal{H} \left( \frac{I -\Phi(z) \Phi(w)^*}{(1-z_1\bar{w}_1) 
(1-z_2\bar{w}_2)} \right) \ \ \text{ and } \ \ H_j : = \mathcal{H} \left( \frac{K_j(z,w)}{1-z_j\bar{w}_j} 
\right), \]
for $j=1,2.$ The Hilbert space $\mathcal{H}_{\Phi}$ \nomenclature{$\mathcal{H}_{\Phi}$}{two-variable de 
Branges-Rovnyak space} is \emph{the two-variable de Branges-Rovnyak space 
associated} to $\Phi$. For $j=1,2,$ define the function $Z_j$ by $Z_j(z):= z_j$. 
Then the $H_j$ Hilbert spaces have the following properties:
\begin{itemize}
\item[(1)]  $Z_j H_j \subseteq H_j$ and multiplication by $Z_j$ on $H_j$ is a contraction.
\item[(2)] The reproducing kernels of the $H_j$ sum to the kernel of $\mathcal{H}_{\phi}$.
\end{itemize}
Basic facts about reproducing kernels imply that if Hilbert spaces $H_1$ and $H_2$ satisfy $(1)$ and $(2)$, then
 the numerators of their reproducing kernels are Agler kernels of
 $\Phi.$ 
\end{rem}

Agler used non-constructive methods to obtain Agler kernels, and a
major stride was made in this theory when Ball-Sadosky-Vinnikov proved
the existence of Agler kernels through constructive Hilbert space
geometric methods.  Indeed, our analysis is motivated by their work on
two-evolution scattering systems and scattering subspaces associated
to $\Phi \in \SEE.$ In \cite{bsv05}, they showed that such scattering
subspaces have canonical decompositions into subspaces $S_1$ and
$S_2$, each invariant under multiplication by $Z_1$ or $Z_2.$ This
work was continued in \cite{gkvw08} where a specific scattering
subspace associated to $\Phi$, denoted $\Kphi$, was used to show that
canonical decompositions of $\Kphi$ yield Agler kernels $(K_1,K_2)$ of
$\Phi$.  The analysis from \cite{bsv05} was also extended in \cite{bkvsv}; 
here, many results from \cite{bsv05} are illuminated or extended 
via the theory of formal reproducing kernel Hilbert spaces.

While more explicit, the approaches so far do not shed much light on
the actual structure of the Hilbert spaces $\mathcal{H}(K_j)$ and the
functions contained therein for general Schur functions.  The spaces
$\mathcal{H}(K_j)$ have been shown to possess a very rich structure
when $\Phi$ is an \emph{inner} function or a \emph{rational inner} function
\cite{bic12, bk12, colwer99, knapde}.  This has led to applications in the study
of two variable matrix monotone functions in \cite{mcc10} and in the
study of \emph{three} variable rational inner functions in
\cite{bk12}.  This structure is also important in the
Geronimo-Woerdeman characterizations of bivariate Fej\'er-Riesz
factorizations as well as the related bivariate auto-regressive filter
problem \cite{gw04}.  The theory is much simpler in these
cases because Agler kernels can be constructed directly from
orthogonal decompositions of $\Hphi$.
% At least in the scalar case,
% Agler kernels can be obtained for non-inner $\Phi$ by approximating
% $\Phi$ by inner functions, however certain details about the Agler
% kernels become unclear in the limit.

Therefore, the major goal of this paper is to show directly that the
rich Agler kernel structure present when $\Phi$ is inner is still
present when $\Phi$ is not an inner function.  A direct application of
this will be to prove that every function in $\SEE$ possesses a
\emph{coisometric} transfer function realization with state space
$\mcH(K_1)\oplus \mcH(K_2)$ for some pair of Agler kernels
$(K_1,K_2)$; this construction answers a question posed by Ball and Bolotnikov
in \cite{bb11}.  We also generalize classical work of Nagy-Foias
connecting regularity of $\Phi \in \mathcal{S}_1(E,E_{*})$ on the
boundary to the regularity of functions in its associated de
Branges-Rovnyak space. See \cite{sar94} for a discussion.

We now outline the rest of the paper.  The structure of $\Hphi$ is
revealed by embedding an isometric copy into the larger scattering
subspace $\Kphi$ alluded to above.  The reader need not know anything
about scattering theory---the basic facts we need are built from
scratch in Section \ref{sect:scattering}.  In Section
\ref{sect:construction}, canonical orthogonal decompositions of $\Kphi$
are projected down to canonical decompositions of $\Hphi$ and these
yield certain pairs of extremal Agler kernels of $\Phi$
denoted
\[ (K^{max}_1, K^{min}_2) \ \ \text{ and } \ \ (K^{min}_1,
K^{max}_2). \]
These pairs are related by a positive kernel $G:\D^2\times \D^2 \to \mathcal{L}(E_*)$
\[
G(z,w) := \frac{K_1^{max}(z,w) - K_1^{min}(z,w)}{1-z_1\bar{w}_1} =
\frac{K_2^{max}(z,w) - K_2^{min}(z,w)}{1-z_2\bar{w}_2}.
\]
In section 4, we show that all Agler kernels of $\Phi$ can be characterized in terms of the
special kernels $K_1^{min}, K_2^{min}, G$:  
\begin{theorem*}[ \ref{thm:maxmin}]
Let $\Phi \in \mathcal{S}_2(E, E_*)$
and let $K_1, K_2: \D^2 \times \D^2 \rightarrow \mathcal{L}(E_*)$. 
Then $(K_1, K_2)$ are Agler kernels of $\Phi$ if and only if there 
are positive kernels $G_1,G_2: \D^2  \times \D^2 \rightarrow \mathcal{L}(E_*)$ 
such that
\[
\begin{aligned} 
K_1(z,w) =& K_1^{min}(z,w) + (1-z_1 \bar{w}_1) G_1(z,w) \\
K_2(z,w) =& K_2^{min}(z,w) + (1-z_2 \bar{w}_2) G_2(z,w) 
\end{aligned}
\]
and $G = G_1 + G_2.$ \end{theorem*}
While Ball-Sadosky-Vinnikov
\cite{bsv05} proved the existence of analogous maximal and minimal
decompositions in the scattering subspace $\Kphi$, our contribution
here is to show that many of these extremality properties also hold
in the space of interest $\Hphi$.  
On the path to our regularity result, we obtain explicit
characterizations of the spaces $\mathcal{H}(K^{max}_j)$ and
$\mathcal{H}(K^{min}_j)$ and use those to show that all
$\mathcal{H}(K_1)$ and $\mathcal{H}(K_2)$ are contained inside ``small'',
easily-studied subspaces of $\Hphi$. Section \ref{sect:functions} has
the details.

In Section 5, we consider applications of this Agler kernel analysis.
When $\Phi$ is square matrix valued, the containments allow us to
characterize when $\Phi$ and the elements of $\mathcal{H}(K_1)$ 
and $\mathcal{H}(K_2)$ extend analytically past portions of $\partial \D^2$, thus
generalizing the regularity result of Nagy-Foias mentioned above.  A
key point is that $\Hphi$ is too big of a space for these
characterizations, and it really is necessary to study Agler kernels
to investigate the regularity of $\Phi$.

We now state the main regularity theorem found in Section
\ref{sect:extensions}.  Let $X \subseteq \mathbb{T}^2$ be an open set
and define the sets
\begin{align*} X_1 & := \left \{ x_1 \in \mathbb{T} : \text{ such that } \exists \ x_2 \text{ with } 
(x_1, x_2) \in X \right \} \\
 X_2 & := \left \{ x_2 \in \mathbb{T} : \text{ such that } \exists \ x_1 \text{ with } (x_1, x_2) \in X \right \} 
 \end{align*}
using $X$  and the sets $\mathbb{E} := \mathbb{C} \setminus \overline{\D}$ and $S := \left \{ 1 / \bar{z}: \det \Phi(z) = 0  \right\}.$
Then, we obtain the following result:
\begin{theorem*}[\ref{thm:extension}] Let $\Phi \in \mathcal{S}_2(E, E_*)$ 
be square matrix valued. Then the following are equivalent:
\begin{itemize}
\item[$(i)$] $\Phi$ extends continuously to 
$X$ and $\Phi$ is unitary valued on $X$.
\item[$(ii)$] There is some pair $(K_1,K_2)$ of 
Agler kernels of $\Phi$ such that the elements 
of $\mathcal{H}(K_1)$ and $\mathcal{H}(K_2)$ 
extend continuously to $X.$
\item[$(iii)$] There exists a domain $\Omega$ containing 
\beq \D^2 \cup X \cup (X_1 \times \D) \cup (\D \times X_2) 
 \cup (\mathbb{E}^2 \setminus S ) \eeq
such that $\Phi$ and the elements of $\mathcal{H}(K_1)$ 
and $\mathcal{H}(K_2)$ extend analytically to $\Omega$ 
for every pair $(K_1, K_2)$ of Agler kernels of $\Phi.$
Moreover the points in the set $\Omega$
are points of bounded evaluation of every $\mathcal{H}(K_1)$
 and $\mathcal{H}(K_2).$ 
\end{itemize}
\end{theorem*}
%The theorem also has applications to the continuation of transfer 
%function realizations past the boundary of $\partial \D^2.$

In Section \ref{sect:tfr}, we return to the setting of transfer function realizations. 
We use the canonical Agler kernels $(K^{max}_1, K^{min}_2)$ to 
construct a T.F.R. of $\Phi$ with refined properties. Specifically we prove:
\begin{theorem*}[\ref{thm:canonicalcmf}] Let $\Phi \in
  \mathcal{S}_2(E,E_*)$ and consider its Agler kernels $(K^{max}_1,
  K^{min}_2).$ Then, $\Phi$ has a coisometric transfer function
  realization with state space $\mathcal{H}(K^{min}_2) \oplus
  \mathcal{H}(K^{max}_1).$ \end{theorem*} 

This construction answers a question posed by Ball and Bolotnikov in \cite{bb11}.  
We also obtain additional information about the block operators $A, B, C,$ and $D$ of the
associated coisometric colligation $U$.  In Section \ref{sect:opkernels}, we provide an
appendix outlining results concerning operator valued reproducing
kernels used in the paper. We supply the commonly used symbols and
table of contents below for convenience.

\newpage

\renewcommand{\nomname}{List of Symbols}
\printnomenclature[1in]

%\newpage

\tableofcontents

\section{Decompositions of Scattering Subspaces} \label{sect:scattering}

For brevity, this paper only outlines the structure of particular
scattering systems defined for $\Phi \in \SEE$. 
Many details of these scattering systems also appear in 
\cite{bsv05} and \cite{bkvsv}.
For a review of the general theory of one- and multi-evolution scattering
systems, see \cite{bsv05}.

\subsection{Notation and Operator Ranges}

Before proceeding to scattering systems, we require some notation.
Let $E$ be a Hilbert space. Then $L^2(E):= L^2(\mathbb{T}^2) \otimes
E$, i.e.  the space of $E$ valued functions on $\T^2$ with square
summable Fourier coefficients. Similarly, $H^2(E) := H^2(\D^2) \otimes
E$ denotes the space of $E$ valued holomorphic functions on $\D^2$
whose Taylor coefficients around zero are square summable. Recall that $Z_1,
Z_2$ denote the coordinate functions $Z_j(z_1,z_2) = z_j$.  We will
define some standard subspaces of $L^2(E)$ according to their Fourier
series support. Let $\ZZ_+ = \{0,1,2,\dots\}$ and $\ZZ_{-} =
\{-1,-2,-3,\dots\}$.  If $N\subset \ZZ^2$ and $f \in L^2(E)$, the
statement $\supp(\hat{f}) \subset N$ means $\hat{f}(n_1,n_2) = 0$ for
$(n_1,n_2) \not \in N$.  Now define
\[
\begin{aligned} 
L^2_{++}(E) &:= \{f \in L^2(E): \supp(\hat{f}) \subset \ZZ_+ \times \ZZ_+ \}\\
L^2_{+\bullet}(E) &:= \{f \in L^2(E): \supp(\hat{f}) \subset \ZZ_+ \times \ZZ \} \\
L^2_{-\bullet}(E) &:= \{f \in L^2(E): \supp(\hat{f}) \subset \ZZ_{-} \times \ZZ\} \\
L^2_{+-}(E) &:= \{f \in L^2(E): \supp(\hat{f}) \subset \ZZ_{+} \times \ZZ_{-}\} \\
L^2_{--}(E) &:= \{f \in L^2(E): \supp(\hat{f}) \subset \ZZ_{-} \times \ZZ_{-}\},
\end{aligned}
\]
and similarly one can define $L^2_{\bullet +}(E),$ $L^2_{\bullet
  -}(E)$, and $L^2_{-+} (E).$ It is well-known that associating an
$H^2(E)$ function $f$ with the $L^2$ function whose Fourier
coefficients agree with the Taylor coefficients of $f$ maps $f$
unitarily to its radial boundary value function in $L^2_{++}(E).$ We
will denote both the function in $H^2$ and the associated function in
$L^2_{++}$ by $f$.

We also require the following definition and simple lemma about
operator ranges; for more details, see the first chapter of
\cite{sar94}.

\begin{defn} Let $\mathcal{K}$ be a Hilbert space and let  $T: \mathcal{K} \rightarrow \mathcal{K}$
 be a bounded linear operator on $\mathcal{K}$. Then the
 \emph{operator range} of T, denoted $\mathcal{M}(T)$,
 \nomenclature{$\mathcal{M}(T)$}{operator range}
 is the Hilbert space consisting of elements in the image of $T$ endowed with the inner product
\[
 \LL T x, Ty \RR_{\mathcal{M}(T)} := \LL P_{(\ker T)^{\perp} }  x,  y \RR_{\mathcal{K}} 
\qquad \forall \ x,y \in \mathcal{K}.
\]
\end{defn}

\begin{lem} \label{lem:oprange} Let $\mathcal{K}$ be a Hilbert space and let  $T: \mathcal{K} 
\rightarrow \mathcal{K}$ be a bounded linear self-adjoint operator on $\mathcal{K}$. 
Then the operator range $\mathcal{M}(T)$ is the closure of the image of $T^2$ in the 
$\mathcal{M}(T)$ norm and $\LL T x, T^2 y \RR_{\mathcal{M}(T)} = \LL T x, y 
\RR_{\mathcal{K}},$ for all $x,y \in \mathcal{K}.$ 
\end{lem}

\begin{proof} We show that if  $\eta \in \mathcal{M}(T)$ and $\eta \perp T^2 \mathcal{K}$,
 then $\eta \equiv 0.$ Fix such an $\eta$ and choose $x \in ( \ker T )^{\perp}$ such that $Tx 
=\eta.$ Then, for each $y \in \mathcal{K}$, 
\[ 0 = \LL \eta, T^2 y \RR_{\mathcal{M}(T)} = \LL x, T y \RR_{\mathcal{K}} =
 \LL Tx, y \RR_{\mathcal{K}} =\LL \eta, y \RR_{\mathcal{K}}  ,\]
which implies $\eta \equiv 0.$ Moreover, for any $x,y \in \mathcal{K}$, 
\[ \LL T x, T^2 y \RR_{\mathcal{M}(T)} = \LL P_{(ker T)^{\perp}} x,  Ty  \RR_{\mathcal{K}} 
= \LL  T P_{(ker T)^{\perp}} x,  y  \RR_{\mathcal{K}} = \LL T x, y \RR_{\mathcal{K}}, \]
as desired. \end{proof}

\begin{ex} \label{ex:hphi} Let $\Phi \in \SEE$. The two-variable de Branges-Rovnyak space 
$\Hphi$ is also the operator range of the bounded linear self adjoint operator
\[ \Dphi : =( I - \Phi P_{H^2(E)} \Phi^*)^{1/2}: H^2(E_*) \rightarrow
H^2(E_*). \nomenclature{$\Dphi$}{The operator $(I - \Phi P_{H^2(E)} \Phi^*)^{1/2}$}\]
To see this notice first that by Lemma \ref{lem:oprange},
$\Dphi ^2 H^2(E_*)$ is dense in $\mathcal{M}(\Dphi)$ and 
\[ \LL \Dphi f, \Dphi^2 g
\RR_{\mathcal{M}(\Dphi)} = \LL \Dphi f, g \RR_{H^2(E_*)}\] 
for all $f,g \in
H^2(E_*).$ Let $k_z$ be the Szeg\H{o} kernel on the bidisk.  Then, the
reproducing kernel of $H^2(E_{*})$ is $k_z\otimes I_{E_*}$.  Given $f
\in \mathcal{M}(\Dphi)$, $z \in \D^2, v \in E_{*}$, we see that
\[
\LL f, \Dphi^2 k_z v \RR_{\mathcal{M}(\Dphi)} = \LL f, k_z v \RR_{H^2(E_{*})} = \LL
f(z), v \RR_{E_{*}}
\]
and therefore the operator range of $\Dphi$ is a reproducing kernel
Hilbert space on $\D^2$ with reproducing kernel
\[
\frac{I-\Phi(z)\Phi(w)^*}{(1-z_1\bar{w}_1)(1-z_2\bar{w}_2)}.
\]
Specifically, $\mathcal{M}(\Dphi)$ is equal to the de Branges-Rovnyak space associated to
$\Phi$, which is $\Hphi.$  This follows from the standard identity for reproducing
kernels $P_{H^2} \Phi^* k_z v = \Phi(z)^* k_z v$ and the computation
$\Dphi^2 k_z v = (I-\Phi P_{H^2} \Phi^*) k_z v = k_z v -\Phi
\Phi(z)^*k_z v$.
 \end{ex}

The following consequence of Douglas's lemma \cite{do66} is found on
page 3 of \cite{sar94}.
 
\begin{lem} \label{lem:douglas}
Let $\mcK$ be a Hilbert space and let $A:\mcK \to \mcK, B:\mcK \to
\mcK$ be bounded linear operators.  Then, $\mathcal{M}(A) =
\mathcal{M}(B)$ if and only if $AA^* = BB^*$.  
\end{lem}

\subsection{The de Branges-Rovynak Models}

Now we proceed to scattering systems: 

\begin{defn} A \emph{two-evolution scattering system} $\mathcal{S} = (\scrH, \mathcal{U}_1, 
\mathcal{U}_2, \mathcal{F}, \mathcal{F}_*)$ consists of a Hilbert space $
\scrH$, two unitary operators $\mathcal{U}_1$, $\mathcal{U}_2: \scrH \rightarrow 
\scrH,$ and two wandering subspaces $\mathcal{F}, \mathcal{F}_* \subseteq 
\scrH$ of $\mathcal{U}_1$ and $\mathcal{U}_2$, i.e.
\[ 
\mathcal{F} \perp \U_1^{n_1} \U_2^{n_2} \F \ \ \text{ and }
 \ \ \ \mathcal{F}_* \perp \U_1^{n_1} \U_2^{n_2} \F_* \ \qquad \forall \ (n_1, n_2) \in \mathbb{Z}^2 
\setminus(0,0).
\]
 \end{defn}

Given any $\Phi  \in \SEE$, one can define the de Branges-Rovnyak model for $\Phi$.  
This is a concrete transcription of the (almost) unique minimal scattering 
system whose scattering function coincides with $\Phi.$  See \cite{bsv05} for the proof and additional theory. 

\begin{defn}The \emph{de Branges-Rovnyak model for $\Phi \in \SEE$} consists of the operator 
range, denoted $\scrH$, \nomenclature{$\scrH$}{de Branges-Rovnyak
  model for $\Phi$} of the following bounded linear self-adjoint operator:
\[
\left[ \begin{array}{cc} I & \Phi \\
\Phi^*& I \end{array} \right]^{1/2} :
\left[ \begin{array}{c}
L^2(E_*) \\
L^2(E) \end{array} \right] \rightarrow \left[ \begin{array}{c}
L^2(E_*) \\
L^2(E) \end{array} \right]. 
\]
Then $\scrH$ has inner product given by
\[
\left \langle  \left[ \begin{array}{cc} I & \Phi \\
\Phi^*& I \end{array} \right]^{1/2} 
\left[ \begin{array}{c}
f \\
g \end{array} \right] ,
 \left[ \begin{array}{cc} I & \Phi \\
\Phi^*& I \end{array} \right]^{1/2}
\left[ \begin{array}{c}
f' \\
g' \end{array} \right] \right \rangle_{\scrH} 
:= \left \langle P_{Q^{\perp}}  \left[ \begin{array}{c}
f \\
g \end{array} \right], \left[ \begin{array}{c}
f' \\
g' \end{array} \right] \right \rangle_{L^2(E_*) \oplus L^2(E)},
\]
where $Q = \ker  \left[ \begin{array}{cc} I & \Phi \\
\Phi^*& I \end{array} \right]^{1/2}.$ Lemma \ref{lem:oprange} 
implies the image of the operator $\op$ is dense in $\scrH$ 
and that
\[ 
\left \langle  \vecf, \op \vecft \right \rangle_{\scrH} 
= \left \langle  \vecf, \vecft  \right \rangle_{L^2(E_*) \oplus L^2(E)},\qquad \forall 
\ \vecf \in \scrH.
\]
The de Branges-Rovnyak model also contains the following two 
subspaces of $\scrH$:
\[ 
\F: = \left[ \begin{array} {c}
\Phi \\
I \end{array} \right] E =  \left[ \begin{array}{cc} I & \Phi \\
\Phi^*& I \end{array} \right] 
\left[ \begin{array}{c} 0 \\ E \end{array} \right]
\ \ \text{ and } \ \ 
\F_*: = \left[ \begin{array} {c}
I \\
\Phi^* \end{array} \right] E_* =  \left[ \begin{array}{cc} I & \Phi \\
\Phi^*& I \end{array} \right] 
\left[ \begin{array}{c} E_* \\  0 \end{array} \right]\]
and the two operators $\U_1, \U_2: \scrH \rightarrow \scrH$ 
defined by
\[
 \U_j := \left[ \begin{array}{cc}
Z_j I_{E_*} & 0 \\
0 & Z_j I_E \end{array} \right] \qquad \text{ for } j=1,2. 
\]
Each $\U_j$ is onto since
\[ \U_j \op^{1/2} = \op^{1/2} \U_j \ \ \text{ and }
 \ \ \U_j \Big( L^2(E_*) \oplus L^2(E) \Big)  = L^2(E_*) \oplus L^2(E). 
 \]
To see that $\U_j$ is isometric, observe that $\U_j$ 
preserves the $\scrH$ norm on the image of $\op$ since:
\begin{align*} \Bigg \| \U_j \op \vecf \Bigg \|^2_{\scrH} &
= \LL \U_j \op \vecf, \U_j \vecf \RR_{L^2(E_*) \oplus L^2(E)} \\
&= \LL  \op 
\begin{bmatrix}
Z_j f \\
Z_j g \end{bmatrix} ,  
\begin{bmatrix}
Z_j f \\
Z_j g \end{bmatrix} \RR_{L^2(E_*) \oplus L^2(E)} \\
& = \Bigg \|  \op \vecf \Bigg \|^2_{\scrH}. 
\end{align*}
Since said image is dense in $\scrH$, each $\U_j$ is unitary. Observe that $\F$ 
is \emph{wandering} for $\U_1$ and $\U_2$ since if $\eta, \nu \in E$ and $(n_1,n_2) \ne (0,0)$, then
\[
\begin{aligned} \LL \vecphi \eta, \ \U_1^{n_1} \U_2^{n_2} \vecphi \nu \RR_{\scrH}  &= \LL \op 
\left[ \begin{array}{c} 0 \\ \eta  \end{array} \right], \op \left[ \begin{array}{c} 
0 \\ Z_1^{n_1}Z_2^{n_2} \nu  \end{array} \right] \RR_{\scrH} \\
& \\
&= \LL  \op \left[ \begin{array}{c} 0 \\ \eta  \end{array} \right], 
\left[ \begin{array}{c} 0 \\ Z_1^{n_1}Z_2^{n_2} \nu  \end{array} \right] \RR_{L^2(E_*) \oplus L^2(E)} \\
&\\ 
& =  \LL \eta, \  Z_1^{n_1}Z_2^{n_2} \nu \RR_{L^2(E)},
\end{aligned}
\]
which is zero. Analogous arguments show $\F_*$ is wandering. We will usually just write $\U_j = Z_j$, 
unless we wish to emphasize the connection to scattering systems.
\end{defn}

The following remarks detail additional facts about $\scrH$. 

\begin{rem} \label{rem:hspace} \textbf{Alternate Characterization of $\scrH$.} 
Define the bounded linear self-adjoint operators 
\[
\begin{aligned}  \Delta: &= (I - \Phi^*\Phi)^{1/2}: L^2(E) \rightarrow L^2(E) \\
 \Delta_*: &= (I - \Phi \Phi^*)^{1/2}: L^2(E_*) \rightarrow L^2(E_*). \end{aligned}
 \nomenclature{$\Delta, \Delta_{*}$}{ $(I-\Phi^*\Phi)^{1/2}, (I-\Phi\Phi^*)^{1/2}$} 
 \]
By Lemma \ref{lem:douglas}, the factorizations
\[
\begin{bmatrix} I & \Phi \\ \Phi^* &
  I \end{bmatrix}^{1/2} \begin{bmatrix} I & \Phi \\ \Phi^* &
  I \end{bmatrix}^{1/2} =  \begin{bmatrix} I & 0 \\ \Phi^* &
  \Delta \end{bmatrix} \begin{bmatrix} I & \Phi \\ 0 &
  \Delta \end{bmatrix}
 = \begin{bmatrix} \Delta_{*} & \Phi \\ 0 &
   I \end{bmatrix} \begin{bmatrix} \Delta_{*} & 0 \\ \Phi^{*} &
   I \end{bmatrix}
\]
show that
\begin{align}
\scrH &= \mathcal{M} \left( \begin{bmatrix} I & 0 \\ \Phi^* &
  \Delta \end{bmatrix} \right) = \left\{\begin{bmatrix} f \\ g \end{bmatrix}:
    f\in L^2(E_*), g \in L^2(E), g- \Phi^* f \in \Delta L^2(E)
  \right\} \label{eqn:hphichar} \\
&= \mathcal{M} \left( \begin{bmatrix} \Delta_{*} & \Phi \\ 0 &
   I \end{bmatrix} \right) = \left\{\begin{bmatrix} f \\ g \end{bmatrix}:
    f\in L^2(E_*), g \in L^2(E), f- \Phi g \in \Delta_{*} L^2(E_{*})
  \right\}.
\end{align}
where the equality is on the level of Hilbert spaces, not just as
sets.These characterizations of $\scrH$ can be used 
to show that the linear maps 
\beq \vecf \mapsto f \text{ and } \vecf \mapsto g \eeq
are contractive operators from $\scrH$ onto $L^2(E_*)$ and 
$L^2(E)$ respectively. To see this, note that for each element 
in $\scrH$, there is an $h \in L^2(E)$ such that 
\[ \vecf = 
 \begin{bmatrix} I & 0 \\ \Phi^* &
  \Delta \end{bmatrix}
\begin{bmatrix}
f \\
h \end{bmatrix}, \text{ where } 
\begin{bmatrix}
f \\
h \end{bmatrix} \perp \ker
\begin{bmatrix} I & 0 \\ \Phi^* &
  \Delta \end{bmatrix}.
\]
Since $\scrH$ and the operator range of $ \begin{bmatrix} 
I & 0 \\ \Phi^* & \Delta \end{bmatrix}$ coincide as Hilbert spaces, 
\begin{equation} \label{eqn:fnorm}
\left \| \vecf \right \|^2_{\scrH} = \| f \|^2_{L^2(E_*)} + \|h \|^2_{L^2(E)} 
\ge \| f \|^2_{L^2(E_*)}. \end{equation}
Similarly, the equality between $\scrH$ and the operator range of 
$\begin{bmatrix} \Delta_{*} & \Phi \\ 0 &
   I \end{bmatrix}$  shows that for each element of $\mathscr{H}$, 
\begin{equation} \label{eqn:gnorm}  \| g \|_{L^2(E)} \le \left \| 
\vecf \right \|_{\scrH}. \end{equation}
\end{rem}

The following remark discusses additional subspaces 
of $\scrH$ that are important for the structure of the 
scattering system: 

\begin{rem} \label{rem:kphi} \textbf{The Scattering Subspace $\Kphi.$}
  The incoming subspace $\W_*$ and outgoing subspace $\W$
  \nomenclature{$\W_*, \W$}{Incoming and outgoing subspaces} of the de
  Branges-Rovnyak model are defined as follows:
\[
\begin{aligned}
 \W_*&:= \bigoplus_{n \in \ZZ^2 \setminus \mathbb{Z}^2_+} 
\U_1^{n_1} \U_2^{n_2} \F_*  = \vecphis L^2 \ominus H^2(E_*) \\
 \W &: = \bigoplus_{n \in \mathbb{Z}^2_+} \U_1^{n_1} \U_2^{n_2} \ \F = 
 \vecphi H^2(E). \end{aligned}
\]
An easy calculation shows $\W \perp  \W_*$ in $\scrH$. This means 
$\scrH$ decomposes as 
\[ 
\scrH = \W_* \oplus \Kphi \oplus \W, \]
where $\Kphi := \scrH \ominus (\W \oplus \W_*)$ is called the 
\emph{scattering subspace}. \nomenclature{$\Kphi$}{the scattering
  subspace of $\Phi$} A simple computation shows that 
\[ \vecf \perp \W_* \text{ iff }  f \in H^2(E_*) \text{ and } \vecf \perp \W 
 \text{ iff } g \in L^2\ominus H^2(E).
 \]
This means that the scattering subspace
\[
\begin{aligned} \Kphi &: = \scrH \ominus ( \W \oplus \W_*) \\ 
& = \left \{ \vecf \in \scrH : f \in H^2(E_*), \ g \in L^2\ominus H^2(E) \right \}.
\end{aligned}
\]
Using the alternate characterizations of $\scrH$  from Remark 
\ref{rem:hspace}, it follows that
\[
\begin{aligned} \Kphi &  = \left \{ \vecf :  f \in H^2(E_*), \ g \in L^2\ominus H^2(E)
, \ g-\Phi^*f \in \Delta L^2(E) \right \} \\
& = \left \{ \vecf :  f \in H^2(E_*), \ g \in L^2\ominus H^2(E), \  f- \Phi g \in 
\Delta_* L^2(E_*) \right \}.
\end{aligned}
\]
The following operator gives the orthogonal projection onto $\Kphi:$
\[
 P_{\Phi} : = 
\left[ \begin{array}{cc}
P_+ & -\Phi P_+ \\
-\Phi^*P_{-} & P_{-} \end{array} \right],
\] \nomenclature{$P_{\Phi}$}{Projection onto $\Kphi$}
where $P_+ = P_{H^2}$, and 
$P_{-} =P_{L^2 \ominus H^2}$, \nomenclature{$P_{+}, P_{-}$}{Projection onto
  $H^2, L^2\ominus H^2$} for either $L^2 \ominus H^2(E)$ or 
$L^2 \ominus H^2(E_*).$ It is easy to check  that $P_{\Phi}^2 = 
P_{\Phi}$, $P_{\Phi}|_{\Kphi} \equiv I$ and $ P_{\Phi}|_{\W \oplus 
\W_*} \equiv 0.$ 
% The following subspaces of $\scrH$ are also
%  important and will be needed  later:
% \[
% \begin{aligned} \WT_*&:= \bigoplus_{n \in \ZZ^2} \U_1^{n_1} 
% \U_2^{n_2} \ \F_*  = \vecphis L^2(E_*) \\
%  \WT &: = \bigoplus_{n \in \ZZ^2} \U_1^{n_1} \U_2^{n_2} \ \F =
%   \vecphi L^2(E) .\end{aligned}
%  \]  \nomenclature{$\WT, \WT_{*}$}{Certain subspaces of $\scrH$}
\end{rem}  

\begin{rem}[Inner functions]
  When $\Phi$ is an inner function, namely when $\Phi^*\Phi = I,
  \Phi\Phi^* = I$ a.e.~on $\T^2$, the above machinery simplifies
  significantly and scattering systems are not really necessary.  In
  this case, $\Delta = 0, \Delta_{*} = 0$, so that 
  \[ \Kphi =
  \left\{ \begin{bmatrix} f \\ \Phi^* f \end{bmatrix} : f \in
    H^2(E_*), \Phi^*f \in L^2\ominus H^2(E) \right\}. \]
   Evidently, the
  first component in this space is $f \in H^2(E_{*})$ such that
  $\Phi^*f \in L^2\ominus H^2(E)$.  This is equivalent to saying $f
  \in H^2(E_{*}) \ominus \Phi H^2(E)$.  This space is the usual model
  space associated to the inner function $\Phi$; it is studied in
  \cite{bsv05} and is studied in great depth in \cite{bk12}.  Although
  in this paper we recover many results from \cite{bk12}, there are
  many results related to rational inner functions in \cite{bk12} that
  we do not mention here. In general, the paper \cite{bk12} is
  a more accessible introduction to the present material.
\end{rem}

\subsection{Decompositions of $\Kphi$}

In \cite[Theorem 5.5]{bsv05}, Ball-Sadosky-Vinnikov prove 
the following canonical decomposition of $\Kphi.$ For completeness, 
we include a simple proof here as well.

\begin{thm} \label{thm:kdecomp} Define these subspaces of 
the scattering subspace $\Kphi$:
\[
\begin{aligned}
 S^{max}_1 &= \left \{ \vecf \in \Kphi: Z_1^k \vecf \in \Kphi \ \forall \ 
 k \in \NN \right \} \ 
 S^{min}_1 = \text{closure }P_{\Phi} \vecphi L^2_{+-}(E) \\
 S^{max}_2 &= \left \{ \vecf \in \Kphi: Z_2^k \vecf \in \Kphi \ \forall  \ 
 k \in \NN \right \} \ 
 S^{min}_2 = \text{closure }P_{\Phi} \vecphi L^2_{-+}(E),  
\end{aligned}
\nomenclature{$S_j^{max}, S_j^{min}$}{Subspaces of the scattering subspace}
\] 
where each closure is taken in $\Kphi.$ Then, each $S^{max}_j$
 and $S^{min}_j$ is  invariant under multiplication by $Z_j$ and 
\begin{align}  \label{eqn:kdecomp} \Kphi = S^{max}_1 \oplus 
S^{min}_2 = S^{min}_1 \oplus S^{max}_2.  \end{align}
\end{thm}

\begin{proof} Our first observation is that $S_1^{max}$ is equal to
\begin{multline*}
\left(\begin{bmatrix} I \\ \Phi^* \end{bmatrix} L^2\ominus
  H^2(E_*)\right)^{\perp} \cap \left(\begin{bmatrix} \Phi \\
    I \end{bmatrix} L^2_{\bullet+}(E)\right)^{\perp} \\
=
\left\{ \vecf \in \scrH: f \in H^2(E_*), g \in
  L^2_{\bullet-}(E)\right\}
\end{multline*}
since $Z_1^k \vecf \perp \begin{bmatrix} \Phi \\ I \end{bmatrix}
H^2(E)$ for all $k\geq 0$ if and only if $\vecf \perp \begin{bmatrix}
  \Phi \\ I \end{bmatrix} L^2_{\bullet+}(E)$, which is equivalent to
saying $g \in L^2_{\bullet-}(E)$.  Therefore, $S_1^{max}$ is equal to
\begin{multline*}
\left(\begin{bmatrix} I \\ \Phi^* \end{bmatrix} L^2\ominus
  H^2(E_*)\right)^{\perp} \cap \left(\begin{bmatrix} \Phi \\
    I \end{bmatrix} H^2(E)\right)^{\perp} \cap \left(\begin{bmatrix} \Phi \\
    I \end{bmatrix} L^2_{-+}(E)\right)^{\perp} \\
= \Kphi \ominus
P_{\Phi} \left(\begin{bmatrix} \Phi \\  I \end{bmatrix}
  L^2_{-+}(E)\right). 
\end{multline*}
Hence,
\[
\Kphi \ominus S_{1}^{max} = \text{closure } P_{\Phi}
\left(\begin{bmatrix} \Phi \\ I \end{bmatrix} L^2_{-+}(E)\right)
=S_2^{min},
\]
which shows $\Kphi = S_1^{max} \oplus S_2^{min}$ and similarly $\Kphi
= S_1^{min} \oplus S_2^{max}$.  
It is also clear that $S_j^{max}$ is invariant under $Z_j$ for $j=1,2$.
Showing the same is true for $S_j^{min}$ requires more work.  Define
the following subspace of $\scrH$
\[
\Q = \left(\begin{bmatrix} I \\ \Phi^* \end{bmatrix} L^2_{\bullet -}(E_*)
\right)^{\perp} \cap \left(\begin{bmatrix} \Phi \\
    I \end{bmatrix} L^2_{\bullet+}(E)\right)^{\perp}
\]
and notice that $\Q$ is invariant under both $Z_1$ and $\bar{Z}_1$.  
Projection onto $\Q$ is given by
\[
P_{\Q} = \begin{bmatrix} P_{\bullet+} & -\Phi P_{\bullet+} \\
  -\Phi^* P_{\bullet -} & P_{\bullet-} \end{bmatrix}
\]
where $P_{\bullet\pm}$ is projection onto the appropriate
$L^2_{\bullet\pm}$ space; the proof of this fact is similar to the
proof of the formula for $P_{\Phi}$.  Now it can be directly checked
that
\[
P_{\Phi} \left(\begin{bmatrix} \Phi \\ I \end{bmatrix}
  L^2_{+-}(E)\right) = P_{\Q} \left(\begin{bmatrix} \Phi \\
    I \end{bmatrix} L^2_{+-}(E)\right).
\]
The key things to notice are that since $\Phi L^2_{+-}(E) \subset
L^2_{+\bullet}(E_*)$, it follows that $P_{\bullet+} \Phi
L^2_{+-}(E) = P_{+} \Phi L^2_{+-}(E)$, $P_{\bullet+}L^2_{+-} =0 =
P_{+}L^2_{+-}$, $P_{\bullet-} \Phi L^2_{+-}(E) = P_{-} \Phi
L^2_{+-}(E)$, and $P_{\bullet-} L^2_{+-}(E)= P_{-} L^2_{+-}$.
However, since $\Q$ is invariant under $Z_1$ and $\bar{Z}_1$, it
follows that $P_{\Q}$ commutes with $Z_1$.  Since $\begin{bmatrix}
  \Phi \\ I \end{bmatrix} L^2_{+-}(E)$ is invariant under $Z_1$, we
see that 
\[
P_{\Q} \left(\begin{bmatrix} \Phi \\
    I \end{bmatrix} L^2_{+-}(E)\right)
\]
is invariant under $Z_1$, and hence so is its closure.  This shows
$S_1^{min}$ is invariant under $Z_1$ and the proof that $S_2^{min}$ is
invariant under $Z_2$ is similar. \end{proof}

\begin{defn} \label{defn:residspace} 
\textbf{The Residual Subspace $\mcR$.} 
It is also useful to consider the residual subspace 
$\mcR$ of $\Kphi$ defined initially as $ \mcR:= 
S^{max}_1 \ominus S^{min}_1.$ Using the decomposition 
in (\ref{eqn:kdecomp}), it is basically immediate that
\[ \mcR = S^{max}_2 \ominus S^{min}_2 = S^{max}_1
 \cap S^{max}_2. \]
\end{defn} \nomenclature{$\mcR$}{The residual subspace of the
  scattering subspace}

\section{Constructing Agler Decompositions}\label{sect:construction}

\subsection{Connections between $\Kphi$ and $\Hphi$} 

The decompositions of $\Kphi$ into $S^{max}_j$ and 
$S^{min}_j$ can be used to construct similar decompositions
 of $\Hphi.$ The following results link $\Kphi$ and $\Hphi$.

\begin{lem} \label{lem:isom} There is an isometry $V: \Hphi \rightarrow \Kphi$ such that 
\[
\begin{aligned} Vf  =  \vecf   \text{ for some } g \in L^2 \ominus H^2(E)  \text{ and } 
V^* \vecf   = f  \ \ \forall g \text{ with } \vecf \in \Kphi.
\end{aligned}
\]
\end{lem} \nomenclature{$V$}{Canonical isometry from $\Hphi$ to $\Kphi$}
 
\begin{proof} As was mentioned in Example \ref{ex:hphi}, the set 
$\Dphi^2 H^2(E_*)$ is dense in $\Hphi$. 
Define the operator $V$ on $\Dphi^2 H^2(E_*)$ by
\[
 V \Dphi^2 h = P_{\Phi} \vecphis h \qquad \forall \ h \in H^2(E_*).
\]
Notice that this equals
\[
\begin{bmatrix} P_{+} & -\Phi P_{+} \\ -\Phi^* P_{-} &
  P_{-} \end{bmatrix} \begin{bmatrix} I \\ \Phi^* \end{bmatrix} h =
\left[ \begin{array}{c} 
\Dphi^2 h \\
P_- \Phi^* h \end{array}
\right] 
=   \op \begin{bmatrix} h \\
  -P_{+} \Phi^* h \end{bmatrix}.
\]
The computation
\[
\begin{aligned}
\left \|\op \begin{bmatrix} h \\
  -P_{+} \Phi^* h \end{bmatrix} \right\|^2_{\scrH} &= \ip{ \op \begin{bmatrix} h \\
  -P_{+} \Phi^* h \end{bmatrix}}{ \begin{bmatrix} h \\
  -P_{+} \Phi^* h \end{bmatrix}}_{L^2(E_*) \oplus L^2(E)} \\
&=
\ip{\begin{bmatrix} \Dphi^2 h \\ P_{-} \Phi^* h\end{bmatrix}}{\begin{bmatrix} h \\
  -P_{+} \Phi^* h \end{bmatrix}}_{L^2(E_*) \oplus L^2(E)} \\
&= \ip{\Dphi^2 h}{h}_{L^2(E_{*}) }\\
&= \|\Dphi^2 h\|^2_{\Hphi} 
\end{aligned}
\]
at once shows that $V$ is well-defined ($\Dphi^2 h=0$ implies
$V\Dphi^2 h = 0$) and isometric, and therefore extends to an isometry
from $\Hphi$ to $\Kphi$.  To see that the first
component of $Vf$ is always $f$, it suffices to notice that since
the projection $\pi:\vecf \mapsto f$ is bounded from $\scrH$ to
$L^2(E_*)$ and since we have $\pi V f = f$ for the dense set of $f
\in \Dphi^2 H^2(E_*)$, the identity $\pi V f = f$ must hold for
all $f\in \Hphi$ by boundedness of $\pi V$.

Now, $V^*$ is a partial isometry from $\K_{\phi}$ onto $\Hphi$,
and  
\[
 \text{ker } V^* =(\text{range } V)^{\perp} = \left\{ \begin{bmatrix} 0
  \\ g \end{bmatrix} : g \in L^2\ominus H^2(E) \cap \Delta L^2(E)\right\}.
\]
The latter equality can be seen from the following computation.  If
$\vecf \in \Kphi$ is orthogonal to the range of $V$, then for any $h \in
H^2(E_*)$
\[
\begin{aligned}
0=\ip{\op \begin{bmatrix} h \\ -P_{+} \Phi^* h \end{bmatrix}}
{\begin{bmatrix}  f\\ g \end{bmatrix}}_{\Kphi} &=
 \ip{\begin{bmatrix} h \\ -P_{+} \Phi^* h \end{bmatrix}}
 {\begin{bmatrix}  f\\ g \end{bmatrix}}_{L^2(E_*)\oplus L^2(E)}\\
& = \ip{h}{f}_{L^2(E_*)},
\end{aligned}
\]
since $f \in H^2(E_*)$ and $g \in L^2\ominus H^2(E).$ Upon
setting $h=f$, this yields $f=0$.  On the other hand, the above
computation shows that if $\begin{bmatrix} 0 \\ g \end{bmatrix} \in
\Kphi,$ then this element is orthogonal to the range of $V$.  So, the
action of $V^*$ on $\Kphi$ can be directly computed as follows.  Any
$\vecf \in \Kphi$ can be written as $Vf + \begin{bmatrix} 0
  \\ h \end{bmatrix}$ for some $h \in L^2\ominus H^2(E) \cap \Delta L^2(E)$.
Then, $V^* \vecf = f$.  \end{proof}

An immediate corollary of the above theorem is:

\begin{cor} As sets, $\Hphi =  \left \{ f \in H^2(E_*): \text{ there is a }  g  \text{ with } \vecf \in 
\Kphi \right \}.$ \end{cor}

\subsection{Hilbert Spaces in $\Hphi$}

Using the partial isometry $V^*$ and the decompositions of $\Kphi$
given in Theorem \ref{thm:kdecomp}, we can construct Hilbert spaces
yielding Agler decompositions. First, we make some general
observations. Let $K$ be a closed subspace of $\Kphi$, and denote the
operator range of $V^*|_K$ by $H_K$.  
\nomenclature{$H_K$}{Operator range of $V^*\mid_{K}$}
 Then, $f \in H_K$ if and only if
there exists $g$ such that $\vecf \in K$.  Essentially by the definition
of operator range, $V^*\mid_{K}$ is a unitary from $K\ominus (K\cap
\ker V^*)$ onto $H_K$, and the inverse of this unitary will be of the
form $f \mapsto \vecs{f}{A_K f}$ where $A_K:H_K \to L^2(E)$
is some linear operator.  \nomenclature{$A_K$}{Component of isometry from
  $H_K$ into $K$}
By \eqref{eqn:gnorm}, $A_K$ is contractive, i.e.
:
\begin{equation} \label{eqn:genhk}
\|A_K f\|_{L^2(E)} \leq \left\| \vecs{f}{A_K f}\right\|_{\scrH} =
\|f\|_{H_{K}}
\end{equation}
and it is worth pointing out the following representation of the norm
\[
\|f\|_{H_K} = \min \left\{ \left\|
\vecf \right\|_{\scrH}: g \text{ satisfies } \vecf \in K\right\}.
\]
Let 
\[
k_w (z) = \frac{I}{(1-z_1\bar{w}_1)(1-z_2\bar{w}_2)} 
\]
be the Szeg\H{o} kernel on $H^2(E_{*})$. 
\begin{lem} \label{lem:opkern}
The reproducing kernel for $H_K$ is given by
\[
V^* P_{K} V \Dphi^2 k_w(z).
\]
Moreover, if $K$ is an orthogonal direct sum, $K =
\bigoplus_{j=1}^{\infty} K_j$, then the reproducing kernel for $H_K$
is the sum of the reproducing kernels for $H_{K_j}$.
\end{lem}
\begin{proof}
Take any $f\in H_K$; this means
$f = V^* \vecf$, for some $\vecf \in K\ominus \left [ K\cap \ker V^*
  \right]$.  Then, for $w\in \D^2$ and $v \in E_*$
\[
\begin{aligned}
\ip{f}{V^* P_{K} V \Dphi^2 k_w v}_{H_K} &= \ip{\vecf}{V \Dphi^2 k_w
  v}_{\Kphi} \\
&= \ip{V^* \vecf}{\Dphi^2 k_ w v}_{\Hphi} \\
&= \ip{f}{k_w v}_{H^2(E_*)} =
\ip{f(w)}{v}_{E_*}.
\end{aligned}
\]
The assertion about direct sums follows from noticing $P_K =
\sum_{j=1}^{\infty} P_{K_j}$ in the strong operator topology.
\end{proof}

The Hilbert spaces of primary interest are defined as follows:

\begin{defn} Define the Hilbert spaces $H^{max}_j$ and 
$H^{min}_j$ to be the operator ranges of $V^*|_{S^{max}_j}$
 and $V^*|_{S^{min}_j}$. Then
\[ f \in H^{max}_j \text{ if and only if } \exists \ g \text{ with } 
\vecf \in S^{max}_j, 
\]
and the $H^{max}_j$ norm is given by
\[
\| f \|_{H^{max}_j}  := \left  \| P_{S^{max}_j \ominus \left
 [S^{max}_j \cap \ker V^* \right ]} \vecf  \right \|_{S^{max}_j}
= \min \left \{ \left \| \begin{bmatrix}
f \\
\tilde{g} 
\end{bmatrix} \right \|_{S^{max}_j}  : \begin{bmatrix}
f \\
\tilde{g} 
\end{bmatrix} \in S^{max}_j \right \}. 
\]
\end{defn} \nomenclature{$H_j^{max}, H_j^{min}$}{Operator ranges of
  $V^*|_{S_j^{max}}, V^*|_{S_j^{min}}$}

\begin{lem}[Wold decompositions]
\[
 S^{max}_j = \bigoplus_{n \in \NN} Z_j^n \big( S^{max}_j \ominus Z_j S^{max}_j \big )
 \ \oplus M_j^{max}
\]
\[
 S^{min}_j = \bigoplus_{n \in \NN} Z_j^n \big( S^{min}_j \ominus Z_j S^{min}_j \big )
 \ \oplus M^{min}_j
\]
where $M_j^{max}, M_j^{min} \subset \ker V^*$.
\end{lem}

\begin{proof}
Since multiplication by $Z_j$ is an isometry on $S^{max/min}_j$, 
the classical Wold decomposition says that $S_j^{max}, S_j^{min}$ can be
decomposed as above where
\[
M_j^{max} = \bigcap_{n\geq 0} Z_j^n S_j^{max} \text{ and } M_j^{min} =
\bigcap_{n\geq 0} Z_j^n S_j^{min}
\]
so the only thing to show is $M_j^{max} \subset \ker V^*$, since
$M_j^{min} \subset M_j^{max}$.  So, if $\vecf \in \bigcap_{n\geq 0}
Z_1^n S_1^{max}$, then $\bar{Z}_1^n f \in H^2(E_*)$ for all $n\geq 0$,
which can only happen if $f=0$.  This shows $\vecf \in \ker V^*$.
\end{proof}
%  On the other hand, 
% \[
% \ker V^* \cap S^{max}_1 = \left\{ \begin{bmatrix} 0 \\ g \end{bmatrix}
%   \in \scrH: g \in L^2_{\bullet,-}(E) \right\}
% \]
% is contained in $\bigcap_{n\geq 0} Z_1^n S_1^{max}$ since
% $L^2_{\bullet,-}(E)$ is invariant under $Z_1$ and $\bar{Z_1}$.  

\begin{lem} \label{lem:hkernels}
Let $K_j^{max}, K_j^{min}$ be the reproducing kernels for the operator ranges
of $V^*\mid_{S_j^{max} \ominus Z_j S_j^{max}}, V^*\mid_{S_j^{min} \ominus Z_j S_j^{min}}$.  
\nomenclature{$K_j^{max}, K_j^{min}$}{Reproducing kernels for
  $H_{S_j^{max}\ominus Z_j S_j^{max}}, H_{S_j^{min}\ominus Z_j S_j^{min}}$} 
Then, the
  reproducing kernels for $H_j^{max}$ and $H_j^{min}$ are given by
\[
\frac{K_j^{max}(z,w)}{1-z_j \bar{w}_j} \text{ and }
\frac{K_j^{min}(z,w)}{1-z_j \bar{w}_j}.
\]
In addition, if $G$ is the reproducing kernel for the operator range
of $V^*|_{\mcR}$, \nomenclature{$G$}{Reproducing kernel for
  $H_{\mcR}$} then
\begin{equation} \label{eqn:KKG}
\frac{K_j^{max}(z,w)}{1-z_j \bar{w}_j} = \frac{K_j^{min}(z,w)}{1-z_j
  \bar{w}_j} + G(z,w).
\end{equation}
\end{lem}

\begin{proof}
We can focus on $H_1^{max}$ which has reproducing kernel $V^*
P_{S_1^{max}} V \Dphi^2 k_w$ by previous remarks.  Let $P_1$ denote
orthogonal projection onto $S_1^{max} \ominus Z_1 S_1^{max}$.  Then,
orthogonal projection onto $Z_1^n(S_1^{max} \ominus Z_1 S_1^{max})$ is
given by $Z_1^nP_1 \bar{Z}_1^n$.  We now claim that the reproducing
kernel for the operator range of $V^*$ restricted to $Z_1^n(S_1^{max}
\ominus Z_1 S_1^{max})$ satisfies
\[
V^*Z_1^n P_1 \bar{Z}_1^n V \Dphi^2 k_{w}v = \bar{w}_1^n Z_1^n V^* P_1
V \Dphi^2 k_{w}v.
\]
Now for $\vecf \in S_1^{max},$ we have $Z_1^n V^* \vecf = V^* Z_1^n
\vecf$. This means $V^* Z_1^n P_1 = Z_1^n V^* P_1$ and so, for any $f
\in \Hphi$, $v \in E_*$,
\[
\begin{aligned}
\ip{f}{V^* Z_1^n P_1 \bar{Z}_1^n V \Dphi^2 k_{w}v}_{\Hphi} &=
\ip{V^* Z_1^n P_1 \bar{Z}_1^n V  f}{\Dphi^2 k_{w}v}_{\Hphi} \\
&= \ip{Z_1^n V^* P_1 \bar{Z}_1^n V f}{\Dphi^2 k_{w}v}_{\Hphi}
\\
&= w_1^n \ip{f}{V^* Z_1^n P_1 V \Dphi^2 k_{w}v}_{\Hphi} \\
&= \ip{f}{\bar{w}_1^n Z_1^n V^* P_1 V \Dphi^2
  k_{w}v}_{\Hphi},
\end{aligned}
\] 
so that $V^*Z_1^n P_1 \bar{Z}_1^n V \Dphi^2 k_{w}v = \bar{w}_1^n Z_1^n
V^* P_1 V \Dphi^2 k_{w}v$.
If we break up $S_1^{max}$ according to its Wold decomposition, then
since $V^*$ annihilates $M_1^{max}$, then Lemma \ref{lem:opkern}
implies that the reproducing kernel of $H_1^{max}$ is given by
\[
\sum_{n \geq 0} \bar{w}_1^n z_1^n V^* P_1 V \Dphi^2 k_{w}(z) =
\frac{V^* P_1 V \Dphi^2 k_{w}(z)}{1-z_1 \bar{w}_1} =
\frac{K_1^{max}(z,w)}{1-z_1\bar{w}_1}.
\]
The formulas for $H_2^{max}$ as well as the $H_j^{min}$ follow
similarly.  The formula \eqref{eqn:KKG} follows from the orthogonal
decomposition $S_j^{max} = S_j^{min} \oplus \mcR$ and Lemma
\ref{lem:opkern}.
\end{proof}

\subsection{Construction of Agler Kernels}

As above, let $K_j^{max}, K_j^{min}$ be the reproducing kernels for
the operator ranges of $V^*|_{S_j^{max} \ominus Z_j S_j^{max}}$ and
  $V^*|_{S_j^{min} \ominus Z_j S_j^{min}}$ respectively. 

%These Hilbert spaces can be used to define Agler kernels of $\Phi$. 

\begin{thm} \label{thm:agdecomp}

%% There exist positive holomorphic
%%  kernels $K^{max}_j, K^{min}_j: \D^2 \times \D^2 \rightarrow
%%  \mathcal{L}(E_*)$ such that 
%% \[
%% H^{max}_j = \mathcal{H} \Bigg ( \frac{K^{max}_j(z,w) }{1-z_j\bar{w}_j} \Bigg) \ \ 
%% \text{ and } \ \ H^{min}_j = \mathcal{H} \Bigg ( \frac{K^{min}_j(z,w) }{1-z_j\bar{w}_j} \Bigg), 
%% \]
%% for $j=1,2.$

The pairs $(K^{max}_1, K^{min}_2)$ and $(K^{min}_1, K^{max}_2)$ are
Agler kernels of $\Phi$, i.e.  for all $z,w \in \D^2,$
\begin{equation} \label{eqn:hdecomp} 
\frac{I_{E_*} - \Phi(z) \Phi(w)^*}{(1-z_1\bar{w}_1)(1-z_2\bar{w}_2)}
\ = \ \frac{K^{max}_1(z,w)}{1-z_1\bar{w}_1} + \frac{K^{min}_2(z,w)}{1-
  z_2\bar{w}_2} \ = \ \frac{K^{min}_1(z,w)}{1-z_1\bar{w}_1} +
\frac{K^{max}_2(z,w)} {1-z_2\bar{w}_2}.
\end{equation}

\end{thm}

\begin{proof} 
The reproducing kernel of $\Hphi$, namely
\[
\Dphi^2 k_w(z) =
\frac{I_{E_*}-\Phi(z)\Phi(w)^*}{(1-z_1\bar{w}_1)(1-z_2\bar{w}_2)}
\]
is the sum of the kernels for $H_1^{max}$ and $H_2^{min}$ by Lemma
\ref{lem:opkern}, and these kernels are given by
\[
V^* P_{S_1^{max}} V \Dphi^2 k_w(z) \text{ and } V^* P_{S_2^{min}} V
\Dphi^2 k_w(z).
\]
By Lemma \ref{lem:hkernels}, these kernels can be computed directly in
terms of the reproducing kernels of $K_1^{max}$ and $K_2^{min}$ to
give us the formula \eqref{eqn:hdecomp}.  
\end{proof}

We remark that by \eqref{eqn:KKG} we get the formula
\[
%\begin{aligned}
\frac{I_{E_*}-\Phi(z)\Phi(w)^*}{(1-z_1\bar{w}_1)(1-z_2\bar{w}_2)} 
%% \frac{1}{1-z_1
%%\bar{w}_1} K_1^{min}(z,w) + \frac{1}{1-z_2 \bar{w}_2}
%%K_2^{max}(z,w) \\
= \frac{K_1^{min}(z,w)}{1-z_1\bar{w}_1}  + \frac{K_2^{min}(z,w)}{1-z_2 \bar{w}_2}
 + G(z,w)
%\end{aligned}
\]
where $G(z,w) = V^* P_{\mcR} V\Dphi^2 k_{w}(z)$ is the reproducing 
kernel of $H_{\mcR}$, the operator range of $V^*|_{\mcR}.$
 
%% An immediate corollary of the calculations in the above proof is the following:
%% \begin{cor} \label{cor:minmax} For $ j=1,2$, 
%% \[ 
%% \begin{aligned} 
%% \mcH(K_j^{min}) &= \{f \in \Hphi:
%% \exists \ g \text{ such that } \vecf \in S_j^{min} \ominus Z_j
%% S_j^{min}\} \\
%% \mcH(K_j^{max}) &= \{f \in \Hphi:
%% \exists \ g \text{ such that } \vecf \in S_j^{max} \ominus Z_j
%% S_j^{max}\}.
%% \end{aligned}
%% \]
%% \end{cor}

%% \begin{proof} This follows from the fact that $K_j^{max/min}$
%%  is the reproducing kernel for the operator range of 
%%  $V^*|_{S_j^{max/min} \ominus Z_j S_j^{max/min}}.$ \end{proof}

\section{General Agler Kernels}

\subsection{Characterizations of General Agler Kernels} \label{sect:characterization}

Assume $(K_1, K_2)$ are Agler kernels of $\Phi \in
 \mathcal{S}_2(E, E_*)$ and define the Hilbert spaces 
 \begin{equation} \label{eqn:hspaces} H_1 : = \mathcal{H}
 \left( \frac{K_1(z,w) }{1-z_1\bar{w}_1} \right) \text{ and }  H_2 
 : = \mathcal{H} \left( \frac{K_2(z,w) }{1-z_2\bar{w}_2} 
\right). \end{equation}
Our goal is to use these auxiliary Hilbert spaces
 $H_1$ and $H_2$ to characterize $(K_1,K_2)$ in terms
of the extremal kernels $K^{max/min}_1$ and $K^{max/min}_2.$ 
The first main result is the following theorem:

\begin{thm} \label{thm:mincontain} Let $\Phi \in 
\mathcal{S}_2(E,E_*)$ and let $(K_1,K_2)$ be Agler kernels
 of $\Phi$. Define $H_1, H_2$ as in \eqref{eqn:hspaces}. Then
\[
H_1 \subseteq H_1^{max} \ \text{ and } \ H_2 \subseteq
 H_2^{max} 
 \]
 and these containments are contractive, i.e. for $j=1,2$
\[ 
\|f \|_{H_j^{max}} \le \| f \|_{H_j} \ \qquad \forall \ f \in  
H_j. 
\]
%% \[
%% H^{min}_ 1 \subseteq H_1 \ \text{ and } \ H^{min}_2 \subseteq
%%  H_2 
%%  \]
%%  and these containments are contractive, i.e. for $j=1,2$
%% \[ 
%% \|f \|_{H_j} \le \| f \|_{H^{min}_j} \ \qquad \forall \ f \in  
%% H^{min}_j. 
%% \]
 \end{thm}

\begin{proof}
Let $f \in H_1$ and assume $\|f\|_{H_1} = 1$.  Then for all $n\geq 0$,
$Z_1^n f \in H_1 \subset \Hphi$ and $\|Z_1^n f\|_{\Hphi} \leq \|Z_1^n
f\|_{H_1} \leq 1$, since multiplication by $Z_1$ is a contraction in
$H_1$.  For each $n$ we can choose $g_n \in L^2\ominus H^2(E)$ such
that $\vecs{Z_1^n f}{g_n} \in \Kphi \ominus \ker V^*$ and
\[\left\|\vecs{f}{\bar{Z}_1^n g_n}\right\|_{\scrH} = \left\|\vecs{Z_1^n f}{g_n}\right\|_{\scrH} =
\|Z_1^n f\|_{\Hphi}
\leq 1.
\]
Notice $F_n:= \vecs{f}{\bar{Z}_1^n g_n} \in \Kphi \ominus \bar{Z}_1^n
\ker V^*,$ since $\bar{Z}_1^n g_n \in \bar{Z}_1^n (L^2\ominus H^2(E))$.
The sequence $\{F_n\}\subset \Kphi$ is bounded in norm and therefore
has a subsequence $\{F_{n_j}\}$ that converges weakly to some $F:=
\vecft$.  We claim that $f=f'$ and $g' \in L^2_{\bullet-}(E)$.  Since
\[
\ip{F_{n_j}}{ \vecphis h}_{\scrH} = \ip{f}{h}_{L^2(E_*)} \to
\ip{F}{\vecphis h}_{\scrH} = \ip{f'}{h}_{L^2(E_*)} \text{ as } j \to
\infty
\]
for all $h \in L^2(E_*)$, we see that $f=f'$.  Next, for any $v \in E$
and $n\in \ZZ, m\geq 0$
\[
\ip{F_{n_j}}{ \vecphi Z_1^n Z_2^mv}_{\scrH} = \ip{\bar{Z}_1^{n_j}
  g_{n_j}}{Z_1^n Z_2^mv}_{L^2(E)} = 0
\]
for $j$ large enough that $n_j + n \geq 0$ since $g_{n_j} \perp
H^2(E)$.  By weak convergence, the above expression converges to
\[
\ip{F}{\vecphi Z_1^n Z_2^mv}_{\scrH} = \ip{g'}{Z_1^nZ_2^m v}_{L^2(E)}
= \ip{\widehat{g'}(n,m)}{v}_{E} = 0
\]
so we see that $g' \perp L^2_{\bullet+}(E)$ and therefore $g' \in
L^2_{\bullet-}(E)$.  Hence we conclude that
\[
F = \vecs{f}{g'} \in S_1^{max}
\]
and so $f = V^{*}F$ must be in $H_1^{max}$.  To show
$\|f\|_{H_1^{max}} \leq 1$, observe that
\[
|\ip{F_{n_j}}{F}_{\scrH}| \to \|F\|^2_{\scrH} 
\]
and 
\[
|\ip{F_{n_j}}{F}_{\scrH}| \leq \|F_{n_j}\|_{\scrH} \|F \|_{\scrH} \leq
\|F\|_{\scrH}
\]
so that $\| F\|_{\scrH} \leq 1$. Finally, $\|f\|_{H_1^{max}} \leq \|
F\|_{\scrH} \leq 1$ as desired.  
Thus, $H_1$ is contractively contained in $H_1^{max}$.\end{proof}

Using the previous result, it is possible to 
characterize all Agler kernels in terms of the 
canonical kernels $K^{min}_1$, $K^{min}_2$ 
and $G$ as follows:

\begin{thm} \label{thm:maxmin} Let $\Phi \in \mathcal{S}_2(E, E_*)$
and let $K_1, K_2: \D^2 \times \D^2 \rightarrow \mathcal{L}(E_*)$. 
Then $(K_1, K_2)$ are Agler kernels of $\Phi$ if and only if there 
are positive kernels $G_1,G_2: \D^2  \times \D^2 \rightarrow \mathcal{L}(E_*)$ 
such that
\[
\begin{aligned} 
K_1(z,w) =& K_1^{min}(z,w) + (1-z_1 \bar{w}_1) G_1(z,w) \\
K_2(z,w) =& K_2^{min}(z,w) + (1-z_2 \bar{w}_2) G_2(z,w) 
\end{aligned}
\]
and $G = G_1 + G_2.$ \end{thm}

\begin{proof}
($\Rightarrow$) Assume $(K_1,K_2)$ are Agler kernels of 
$\Phi$. By Theorem \ref{thm:mincontain} and Theorem \ref{thm:kerdiff},
there are positive kernels $G_1, G_2: \D^2 \times \D^2 \rightarrow
\mathcal{L}(E_*)$ such that
\begin{align*}
G_1(z,w) &= \frac{K_1^{max}(z,w)}{1-z_1\bar{w}_1} -
\frac{K_1(z,w)}{1-z_1\bar{w}_1} =\frac{K_1(z,w)}{1-z_1\bar{w}_1} -
\frac{K^{min}_1(z,w)}{1-z_1\bar{w}_1} \\
G_2(z,w) &=  \frac{K_2^{max}(z,w)}{1-z_2\bar{w}_2} -
\frac{K_2(z,w)}{1-z_2\bar{w}_2}=  \frac{K_2(z,w)}{1-z_2\bar{w}_2} -
\frac{K^{min}_2(z,w)}{1-z_2\bar{w}_2}.
\end{align*}
To show $G_1 + G_2 = G$, recall that since
 $(K_1,K_2)$ are Agler kernels of $\Phi$,
\[
\begin{aligned}
\frac{K_1^{min}(z,w)}{1-z_1 \bar{w}_1 }+  G_1(z,w) +
\frac{K_2^{min}(z,w)}{1-z_2 \bar{w}_2} &+  G_2(z,w)
= \frac{K_1(z,w)}{1-z_1 \bar{w}_1} + \frac{K_2(z,w)}
{1-z_2 \bar{w}_2} \\
&\\
& = \frac{ I_{E_*} - \Phi(z) \Phi(w)^*}{(1-z_1 \bar{w}_1)
(1-z_2 \bar{w}_2)} \\
&\\
& = \frac{K_1^{min}(z,w)}{1-z_1 \bar{w}_1 } +\frac{K_2^{min}
(z,w)}{1-z_2 \bar{w}_2}+  G(z,w),
\end{aligned}
\]
which implies $G = G_1 +G_2.$  \\

($\Leftarrow$) Now assume $(K_1,K_2)$ are positive 
kernels with positive kernels  $G_1,G_2: \D^2 \times
\D^2 \rightarrow \mathcal{L}(E_*)$ satisfying
\[
\begin{aligned} 
K_j(z,w) =& K_j^{min}(z,w) + (1-z_j 
\bar{w}_j) G_j(z,w) 
\end{aligned}
\]
for $j=1,2$ and $G= G_1 +G_2.$  Then 
\[
\begin{aligned} \frac{K_1(z,w)}{1-z_1 \bar{w}_1} + 
 \frac{K_2(z,w)}{1-z_2 \bar{w}_2} 
& =  \frac{K_1^{min}(z,w)}{1-z_1 \bar{w}_1 } +
\frac{K_2^{min}(z,w)}{1-z_2 \bar{w}_2}+  G(z,w) \\
&\\
& =\frac{ I_{E_*} - \Phi(z) \Phi(w)^*}{(1-z_1 
\bar{w}_1)(1-z_2 \bar{w}_2)}, 
\end{aligned}
\]
which implies $(K_1,K_2)$ are Agler kernels of $\Phi.$ 
\end{proof}

\subsection{Containment Properties of $\mathcal{H}(K_1)$
 and $\mathcal{H}(K_2)$} \label{sect:functions}

In this section, we consider the set of functions that can be
contained in $\mathcal{H}(K_1)$ or $\mathcal{H}(K_2).$ This result
generalizes a result about inner functions from \cite{bk12}.  We
require two additional subspaces $\mcR_1$ and $\mcR_2$
of $\scrH$, defined as follows:
\[
\mcR_j = \left \{ \vecf: f \in H^2(E_*), \ g \in Z_j L^2_{-
  -}(E), f-\Phi g \in \Delta_{*} L^2(E_*) \right \}
\] \nomenclature{$\mcR_j$}{Slight enlargements of $\mcR$}
for $j=1,2.$ These are slight enlargements of the residual subspace
$\mcR.$ We can now state the result: 
\begin{thm}  \label{thm:containment} 
Let $\Phi \in \mathcal{S}_2(E,E_*)$.  Then for $j=1,2$
\begin{align*} \Kmax &= \left \{ f : \text{ there exists } g \text{ with } 
\vecf \in \mcR_j \ominus Z_j \mcR \right \} \\ \Kmin &=
\left \{ f : \text{ there exists } g \text{ with } \vecf \in
\mcR_j \ominus \mcR \right \}. \end{align*} If
$(K_1,K_2)$ are general Agler kernels of $\Phi,$ then for $j=1,2$
\[
\begin{aligned}
\mathcal{H}(K_j) &\subseteq \left \{ f : \text{ there exists } g \text{ with } 
\vecf \in \mcR_j  \right \} \\
&= \left \{ f \in H^2(E_*): f\in \big ( \Phi
Z_j L^2_{--}(E) + \Delta_* L^2(E_*) \big ) \right \}. \end{aligned}
\]  \end{thm}

The proof of this result requires several auxiliary results about the
functions in $S^{max}_j \ominus Z_j S^{max}_j$ and $S^{min}_j \ominus
Z_j S^{min}_j.$

\begin{prop} \label{prop:Smin} For $j=1,2$, the following equality holds:
\[
S^{min}_j \ominus Z_j S^{min}_j =   \mcR_j \ominus \mcR.
\]
\end{prop}

\begin{proof} We prove the result for $S^{min}_1.$ We shall make use
  of the proof of Theorem \ref{thm:kdecomp}.  Recall the space $\Q$
  defined there:
\[
\Q = \left(\begin{bmatrix} I \\ \Phi^* \end{bmatrix} L^2_{\bullet -}(E_*)
\right)^{\perp} \cap \left(\begin{bmatrix} \Phi \\
    I \end{bmatrix} L^2_{\bullet+}(E)\right)^{\perp}.
\]
We define and manipulate a related space
\[
\begin{aligned}
\M &= \left(\begin{bmatrix} I \\ \Phi^* \end{bmatrix} L^2_{\bullet -}(E_*)
\right)^{\perp} \cap \left(\begin{bmatrix} \Phi \\
    I \end{bmatrix} L^2\ominus L^2_{--}(E)\right)^{\perp} \\
&=\left(\begin{bmatrix} I \\ \Phi^* \end{bmatrix} L^2_{\bullet -}(E_*)
\right)^{\perp} \cap \left(\begin{bmatrix} \Phi \\
    I \end{bmatrix} L^2_{\bullet+}(E)\right)^{\perp} \cap \left(\begin{bmatrix} \Phi \\
    I \end{bmatrix} L^2_{+-}(E)\right)^{\perp}\\
&= \Q \ominus P_{\Q} \left(\begin{bmatrix} \Phi \\
    I \end{bmatrix} L^2_{+-}(E)\right).
\end{aligned}
\]
Also, note $\M = \left \{ \vecf\in \scrH : f \in L^2_{\bullet +}(E_*), g \in L^2_{- -}(E)\right \}$.
Then, 
\[
\begin{aligned}
\Q \ominus \M &= \text{closure}_{\scrH} P_{\Q} \left(\begin{bmatrix} \Phi \\
    I \end{bmatrix} L^2_{+-}(E)\right) \\
&= \text{closure}_{\scrH} P_{\Phi} \left(\begin{bmatrix} \Phi \\
    I \end{bmatrix} L^2_{+-}(E)\right) = S_1^{min},
\end{aligned}
\]
using the proof of Theorem \ref{thm:kdecomp}.  
Observe that $\mathcal{M} \subseteq Z_1 \mathcal{M} \subseteq
\mathcal{Q}$ and $Z_1 \mathcal{Q} = \mathcal{Q}$. 
Since multiplication by $Z_1$ is an isometry on $\scrH$, we can calculate
\[
\begin{aligned} S^{min}_1 \ominus Z_1 S^{min}_1 
&= \left( \mathcal{Q} \ominus \mathcal{M} \right) \ominus Z_1
\left( \mathcal{Q} \ominus \mathcal{M} \right)\\
& = \left( \mathcal{Q} \ominus \mathcal{M} \right) \ominus 
\left( Z_1 \mathcal{Q} \ominus  Z_1 \mathcal{M} \right)\\
& = \left( \mathcal{Q} \ominus \mathcal{M} \right) \ominus 
\left( \mathcal{Q} \ominus Z_1 \mathcal{M} \right)\\
& = Z_1 \mathcal{M} \ominus  \mathcal{M}.
\end{aligned}
\]
As $S^{min}_1 \ominus  Z_1 S^{min}_1 \subseteq S^{max}_1$, 
we can conclude
\[
\begin{aligned} 
  S^{min}_1 \ominus Z_1 S^{min}_1 =& \big( Z_1 \mathcal{M}
  \cap S^{max}_1 \big) \ominus  \big( \mathcal{M} \cap S^{max}_1 \big) \\
  =& \left \{ \vecf\in \scrH : f \in H^2(E_*), \ g \in Z_1 L^2_{--}(E) \right \}  \\
  &\ominus \left \{ \vecf\in \scrH : f \in H^2(E_*), \ g \in
    L^2_{--}(E) \right \} \\
  =& \mcR_1 \ominus \mcR,
\end{aligned}
\]
as desired. The proof follows similarly for $S^{min}_2.$
\end{proof}

We also obtain similar characterizations of $S^{max}_j \ominus Z_j S^{max}_j$.

\begin{prop} \label{prop:smax} For $j=1,2$ the following equalities hold: 
\[ 
S^{max}_j \ominus Z_j S^{max}_j = \mcR_j \ominus Z_j \mcR.
\]
\end{prop}

\begin{proof} Recall that $S^{max}_j = \mcR \oplus S_j^{min}$ and
  $\mcR, Z_j \mcR \subseteq \mcR_j.$ Now
\[
\begin{aligned}
S_j^{max} &= (S_j^{max} \ominus Z_j S_j^{max}) \oplus Z_j S_j^{max} \\
&= (S_j^{max} \ominus Z_j S_j^{max}) \oplus Z_j \mcR \oplus Z_j
S_j^{min}
\end{aligned}
\]
while $S_j^{max}$ can also be decomposed as 
\[
\begin{aligned}
& \mcR \oplus (S_j^{min} \ominus Z_j S_j^{min}) \oplus Z_j S_j^{min}
\\
=& \mcR \oplus (\mcR_j \ominus \mcR) \oplus Z_j S_j^{min} \\
=& \mcR_j \oplus Z_j S_j^{min}.
\end{aligned}
\]
Together these show $S_j^{max} \ominus Z_j
S_j^{max} = \mcR_j \ominus Z_j\mcR$. 
\end{proof}

% Then we have the following sequence: 
% \[
% \begin{aligned} S^{max}_j \ominus Z_j S^{max}_j &= \big( S^{min}_j 
% \oplus \mcR \big) \ominus Z_j \big( S^{min}_j \oplus \mcR \big)  \\
% &\subseteq \big( S^{min}_j \oplus \mcR \big) \ominus Z_j  S^{min}_j  \\
% &= \big( S^{min}_j \ominus Z_j S^{min}_j \big) \oplus \mcR \\
% & \subseteq \mcR_j.
% \end{aligned}
% \]
% It follows immediately that each
% \[
% \begin{aligned}  S^{max}_j \ominus Z_j S^{max}_j  &= \big( S^{max}_j  
% \cap \mcR_j \big) \ominus \big(Z_j  S^{max}_j  \cap  \mcR_j \big) \\
% & = \mcR_j \ominus Z_j \mcR,
% \end{aligned}
% \]
% as desired.
% \end{proof}

Now we can prove Theorem \ref{thm:containment}. 

\begin{proof} The definitions of $\Kmax$ and $\Kmin$ combined with
Propositions \ref{prop:Smin} and \ref{prop:smax} imply that

\begin{align*} \Kmax &= \left \{ f : \text{ there exists } g \text{ with } 
\vecf \in \mcR_j \ominus Z_j \mcR  \right \} \\
\Kmin &= \left \{ f : \text{ there exists } g \text{ with } 
\vecf \in \mcR_j \ominus \mcR  \right \}, \end{align*}
and then the definition of $\mcR_j$ implies: 
\[ 
\begin{aligned}
\mathcal{H}(K^{max/min}_j ) &\subseteq \left \{ f : \text{ there exists } g \text{ with } 
\vecf \in \mcR_j   \right \} \\
& = \left \{ f \in H^2(E_*): f\in \big ( \Phi Z_j L^2_{--}(E) 
+ \Delta_* L^2(E_*) \big ) \right \}.
\end{aligned}
\]
Now let $(K_1,K_2)$ be any pair of Agler kernels of $\Phi$. 
By Theorem \ref{thm:maxmin}, there are positive kernels 
$G_1, G_2$ such that each
\[
K_j(z,w) = K^{min}_j(z,w)  + (1-z_j\bar{w}_j)G_j(z,w)
\]
and $G= G_1 + G_2.$ This means
\[
 \Big( K^{min}_1 (z,w) + G(z,w) \Big)  - K_1(z,w)  
= G_2(z,w) + z_1 \bar{w}_1 G_1(z,w) 
\]
 is a positive kernel. Similar results hold for $K_2$,  
 so that Theorem \ref{thm:kerdiff} implies $
 \mathcal{H}(K_j)$ is contained  contractively in 
 $\mathcal{H}(K^{min}_j + G).$ But then, Theorem \ref{thm:kersum}
 implies that each $f \in \mathcal{H}(K_j)$ can be 
 written as $f = f_1 + f_2$, for  $f_1 \in \Kmin$ 
 and $f_2 \in \mathcal{H}(G).$ Our above arguments give the 
 desired result for $f_1$ and the definition of $\mathcal{H}(G)$
 gives the desired result for $f_2.$ This means
\[
\begin{aligned}
 \mathcal{H}(K_j) &\subseteq \left \{ f : \text{ there exists } g \text{ with } 
\vecf \in \mcR_j   \right \} \\
& = \left \{ f \in H^2(E_*): f\in \big ( \Phi Z_j L^2_{--}(E) 
+ \Delta_* L^2(E_*) \big ) \right \}.
\end{aligned}
 \]
as desired.
\end{proof}

\section{Applications}

\subsection{Analytic Extension Theorem} \label{sect:extensions}
%\section{Application: Extension Properties} 

%\newcommand{\tDphi}{\widetilde{\Dphi}}
%\newcommand{\tHphi}{\widetilde{\Hphi}}
%\newcommand{\tKphi}{\widetilde{\Kphi}}
 
In this section, we restrict to the situation where $E$ and $E_*$ are
finite dimensional with equal dimensions, so after fixing orthonormal
bases of $E$ and $E_*$, we can assume $\Phi$ is a square matrix of
scalar valued $H^\infty (\D^2)$ functions. The containment
results in Theorem \ref{thm:containment} allow us to give conditions
for when such $\Phi$ and the elements of any $\mathcal{H}(K_1)$ and
$\mathcal{H}(K_2)$ associated to Agler kernels of $\Phi$ extend
analytically past portions of $\partial \D^2$.  We first make some
preliminary comments about defining functions in the canonical spaces
outside of the bidisk.

Any Hilbert space contractively contained in $H^2(E_*)$ clearly has
bounded point evaluations at points of $\D^2$.  On the other hand, for
the spaces $\mcR,\mcR_1,\mcR_2$ we can construct points of bounded
evaluation at certain points of $\E^2$, where $\E = \C \setminus
\overline{\D}$.  Using the notation of \eqref{eqn:genhk}, there is a
unitary map from $H_{\mcR}$ onto $\mcR \ominus (\mcR\cap \ker V^*)$ of
the form
\[
f \mapsto \vecs{f}{A_{\mcR} f}
\]
where $A_{\mcR}$ is a contractive linear map from $H_{\mcR}$ to
$L^2_{--}(E)$.  If $f\in H_\mcR$, then
$\vecs{f}{A_{\mcR} f} \in \scrH$ and so
\[
f = \Phi A_{\mcR} f + (I-\Phi \Phi^*)^{1/2} h \text { by } \eqref{eqn:hphichar}
\]
for some $h \in L^2(E_*)$. Let 
\[
S = \{z \in \E^2: \Phi(1/\bar{z}) \text{ is not invertible} \}.
\]
Since $A_{\mcR} f \in L^2_{--}(E)$, we can
write $A_{\mcR} f = \overline{Z_1 Z_2 g}$ for $g \in H^2(E)$ and then
evaluation at $z \in \E^2 \setminus S$ is defined by
\begin{equation} \label{extdef}
f(z) := (\Phi(1/\bar{z})^*)^{-1} \frac{1}{z_1
  z_2}\overline{g(1/\bar{z})}.
\end{equation}
Since $\D^2$ and $\E^2$ are disjoint, for the moment this is just a
formal definition. However, with additional assumptions on $\Phi$, it
is this definition of $f$ in $\E^2$ that provides a holomorphic
extension of $f$.  This evaluation is bounded since $|g(1/\bar{z})|
\leq C \|g\|_{H^2(E)} = C \|A_{\mcR} f \|_{L^2(E)}$ for some $C>0$ and
then
\[
|f(z)| \leq C \frac{1}{|z_1z_2|}\|(\Phi(1/\bar{z})^*)^{-1}\|
\|A_{\mcR} f\|_{L^2(E)} \leq C
\frac{1}{|z_1z_2|}\|(\Phi(1/\bar{z})^*)^{-1}\| \| f\|_{H_{\mcR}}.
\]
This shows evaluation at $z \in \E^2\setminus S$ is a bounded linear
functional of $H_{\mcR} = \mathcal{H}(G)$.

Analogous analysis can be applied to $\mcR_1,\mcR_2$ so that
$H_{\mcR_1}, H_{\mcR_2}$ possess bounded point evaluations at points
of $\E^2 \setminus S$.   In the case of $f \in H_{\mcR_1}$, since
$A_{\mcR_1}f \in Z_1 L^2_{--}$, we can write $f = Z_1 \overline{Z_1Z_2
  g} = \bar{Z}_2 \bar{g}$ for some $g \in H^2(E_*)$ and then we
replace \eqref{extdef} with
\[
f(z) := (\Phi(1/\bar{z})^*)^{-1} \frac{1}{z_2}\overline{g(1/\bar{z})}
\]
for $z \in \E^2 \setminus S$.  For $H_{\mcR_2}$ we simply switch the
roles of $z_1,z_2$.  Since $\mathcal{H}(K_j^{max/min})$ is
contractively contained in $H_{\mcR_j}$, we can define point
evaluations at points of $\E^2\setminus S$ for the canonical Agler
kernel spaces as well.

We proceed to study analytic extensions of $\Phi$ past the boundary.
Let $X \subseteq \mathbb{T}^2$ be an open set and define the related
sets
\[
\begin{aligned} X_1 & := \left \{ x_1 \in \mathbb{T} : \text{ such that }
 \exists \ x_2 \text{ with } (x_1, x_2) \in X \right \} \\
 X_2 & := \left \{ x_2 \in \mathbb{T} : \text{ such that } \exists \ x_1 
 \text{ with } (x_1, x_2) \in X \right \} .
 \end{aligned}
\]
% and the set
% \[ 
%  S := \left \{ 1 / \bar{z} : \det \Phi(z) = 0  \right\}. 
%  \]
Then we have the following result:
\begin{thm} \label{thm:extension} Let $\Phi \in \mathcal{S}_2(E, E_*)$
  be square matrix valued. Then the following are equivalent:
\begin{itemize}
\item[$(i)$] $\Phi$ extends continuously to 
$X$ and $\Phi$ is unitary valued on $X$.
\item[$(ii)$] There is some pair $(K_1,K_2)$ of 
Agler kernels of $\Phi$ such that the elements 
of $\mathcal{H}(K_1)$ and $\mathcal{H}(K_2)$ 
extend continuously to $X.$
\item[$(iii)$] There exists a domain $\Omega$ containing 
\beq \D^2 \cup X \cup (X_1 \times \D) \cup (\D \times X_2) 
 \cup (\mathbb{E}^2 \setminus S ) \eeq
such that $\Phi$ and the elements of $\mathcal{H}(K_1)$ 
and $\mathcal{H}(K_2)$ extend analytically to $\Omega$ 
for every pair $(K_1, K_2)$ of Agler kernels of $\Phi.$
Moreover the points in the set $\Omega$
are points of bounded evaluation of every $\mathcal{H}(K_1)$
 and $\mathcal{H}(K_2).$ 
\end{itemize}
\end{thm}

\begin{proof} We prove $(i) \Rightarrow (iii) \Rightarrow (ii) \Rightarrow (i).$ 
A similar result for inner functions appears as Theorem 1.5 in \cite{bk12}. Many of the 
arguments in this situation are similar. Thus, we outline the proof and 
provide more details on the points where the two proofs diverge.

Since most of the work occurs in $(i) \Rightarrow (iii),$ let us consider this
implication first. The proof involves $3$ claims. \\

\textbf{Claim 1: $\Phi$ extends analytically to $\Omega.$} \\

Since $\Phi$ extends continuously to $X$ and is unitary valued there, 
there is a neighborhood $W^+ \subseteq \D^2$ such that $\Phi$ is 
invertible on $W^+$ and $X \subseteq \overline{ W^+}.$ Then 
\begin{equation} \label{eqn:phiextend}
\Phi(z): = \left[ \Phi \left ( 1 / \bar{z} \right)^* \right]^{-1}
\end{equation}
defines an analytic function on $\mathbb{E}^2 \setminus S$ that is 
meromorphic on $\mathbb{E}^2$. Define $W^- 
= \left \{ 1 / \bar{z} : z \in W^+ \right \}.$
Then $\Phi$ is analytic on $W^+ \cup W^-$ and continuous on 
$W^+ \cup X \cup W^-.$ 
%%since if $x \in X$, then 
%%\[
%% \lim_{\substack{ z \rightarrow x \\  z \ \in W^-} } \left[ \Phi \left ( 1 / 
%%\bar{z} \right)^* \right]^{-1} = \lim_{\substack{ 1/ \bar{z} \rightarrow 
%%1/ \bar{x} \\  1/\bar{z} \ \in W^+} }  \left[ \Phi \left ( 1 / \bar{z} \right)^*
%% \right]^{-1} = \left[ \Phi \left ( x \right)^* \right]^{-1} 
%%= \Phi(x).
%%\]
By Rudin's continuous edge-of-the-wedge theorem, which appears as 
Theorem A in \cite{rudeow}, there is a domain $\Omega_0$  containing 
$W^{+}\cup X \cup W^{-},$
where $\Phi$ extends analytically. This domain only depends on $X, W^{\pm}.$
Also $\Phi$ is already holomorphic on $\D^2$, meromorphic on $\mathbb{E}^2$,
and holomorphic on $\mathbb{E}^2 \setminus S$ using definition (\ref{eqn:phiextend}).

We can extend this domain further using Rudin's Theorem 4.9.1 in \cite{rud69}. 
It roughly says that if a holomorphic function $f$ on
$\D^2$ extends analytically to a neighborhood $N_x$ of some
$x=(x_1,x_2) \in \mathbb{T}^2,$ then $f$ extends analytically to an open set
containing $\{x_1\} \times \D$ and $\D \times \{x_2\}.$ As the
edge-of-the-wedge theorem guarantees $\Phi$ extends to a neighborhood
$N_x$ of each $x \in X$, Rudin's Theorem 4.9.1 implies $\Phi$
extends analytically to an open set $\Omega_1$ containing $(X_1 \times \D )
\cup (\D \times X_2).$ The \emph{proof} of
Theorem 4.9.1 implies that $\Omega_1$ only depends on the
$\{N_x\}_{x \in X}$. Thus, $\Phi$ extends analytically to
\[ \Omega := \D^2 \cup \Omega_1 \cup \Omega_0  \cup 
\left( \mathbb{E}^2 \setminus S \right). \]

\textbf{Claim 2: Elements of $\mathcal{H}(K_1)$ and $\mathcal{H}(K_2)$ 
extend analytically to $\Omega.$}\\

Let $(K_1,K_2)$ be Agler kernels of $\Phi$ and let $f \in \mathcal{H}(K_1).$
 By the containment result in Theorem  \ref{thm:containment},
\[
 f =\Phi A_{\mcR_1}f  + (I - \Phi \Phi^*)^{1/2} h,
 \]
for some $h \in L^2(E_*)$ and $A_{\mcR_1}f \in Z_1 L^2_{--}(E).$ Then $g:=
\overline{Z_2 A_{\mcR_1} f} 
\in H^2(E),$ and we can define  $f$ analytically on $\mathbb{E}^2 \setminus S$
as before:
\[ f(z) = \Phi(z) \frac{1}{z_2} \overline{g(1/\bar{z})}. 
\]
Then $f$ is analytic on $W^+ \cup W^-$ and $f=\Phi A_{\mcR_1}f $ on $X$. 
%%For $r=(r_1,r_2),$ define
%%\beq f_r(z): = f(r_1 z_1, r_2z_2). \eeq
As in the proof of Theorem 1.5 in \cite{bk12}, we can use the
distributional edge-of-the-wedge theorem, which appears as Theorem B in \cite{rudeow}, 
to extend $f$ to $\Omega_0.$
As before, by an application of  Rudin's Theorem 4.9.1 in \cite{rud69}, we can analytically 
extend $f$ to $\Omega_1$, the set containing $X_1 \times \D$ and
$\D \times X_2$ mentioned earlier.  As $f$ is already holomorphic in 
$\D^2 \cup (\mathbb{E}^2 \setminus S),$ we can conclude that every 
$f \in \mathcal{H}(K_1)$ is holomorphic in $\Omega$.\\

\textbf{Claim 3: Points in $\Omega$ are points of bounded evaluation
  in $\mathcal{H}(K_1)$ and $\mathcal{H}(K_2).$}\\

The proof for inner functions given in \cite{bk12}
essentially goes through to give bounded point evaluations in
$\Omega$.  Recall from the previous section that points of $\D^2$
and $\E^2\setminus S$ are points of bounded evaluation for $\mathcal{H}(K_1)$
and $\mathcal{H}(K_2)$. The next step is to show that the set of points of bounded evaluation is relatively
closed in $\Omega$. This follows using the uniform boundedness principle as in \cite{bk12}.
%%Let $B$ be the set of bounded point evaluations
%%of $\mathcal{H}(K_j)$ in $\Omega$.  We have already remarked that $\D^2\cup
%%\E^2\setminus S \subseteq B$.
%%
%%Suppose $\{w^n\} \subset B$ and $ \{w^n\} \to w \in \Omega$.  As the
%%functions in $\mathcal{H}(K_j)$ are holomorphic in $\Omega,$
%%\[
%%\sup_{n \in \NN } \|f(w^n) \|_{E_*} < \infty \qquad \forall \ f \in \mathcal{H}(K_j).
%%\]
%% By the uniform boundedness principle, there is a constant $M$ such
%% that
%%\[
%%\| f(w^n) \|_{E_*}  \le M \| f \|_{\mathcal{H}(K_j)} \qquad \forall \  f \in \mathcal{H}(K_j)
%%\]
%%and $n \in \NN.$ As each $f$ is holomorphic in
%%$\Omega$, $\{f(w^n)\} \rightarrow f(w)$ and so
%%\[
%% \| f(w) \|_{E_*} \le M \|f \|_{\mathcal{H}(K_j)}  \qquad \forall \ f \in \mathcal{H}(K_j).
%%\] 
%%Hence, $w \in B$ and  $B$ is a relatively closed subset of $\Omega$. 
%%It follows that $B$ contains  $X \cup (X_1 \times \D) \cup (\D \times X_2).$ 
%%\\
To show evaluation at points of $\Omega_0$ are bounded, we merely note
as we did in \cite{bk12} that the proof of the edge-of-the-wedge
theorem in \cite{rudeow} produces the extended values via an integral 
over a compact subset $K$ of $W^{+}\cup X \cup W^{-}$. Since
evaluation at any point of $K$ is bounded in $\mathcal{H}(K_j)$ and since
elements of $\mathcal{H}(K_j)$ are analytic in a neighborhood of $K$,
\[
\sup\{ \|f(z)\|_{E_*}: z \in K\} < \infty
\]
for each $f \in \mathcal{H}(K_j)$ and therefore by the uniform boundedness
principle there exists $M$ such that
\[
\|f(z)\|_{E_*} \leq M \|f\|_{\mathcal{H}(K_j)} \qquad \forall \ f \in \mathcal{H}(K_j)
\]
and $z \in K$.  So, since values of $f$ in $\Omega_0 $
 are given by an integral of $f$ over $K$, it follows that evaluation
at points in $\Omega_0$ are bounded in $\mathcal{H}(K_j)$. 
Now consider the points in $\Omega_1.$ As Rudin's Theorem 4.9.1 in \cite{rud69}
also constructs the extension of $f$ using values of $f$ at points in compact sets $K \subset \Omega_0$,
the uniform boundedness principle implies that the points in $\Omega_1$ are also
points of bounded evaluation.\\

%%Now consider the points of $\Omega_1,$ the set guaranteed by Proposition
%%\ref{prop:rudin2}. As we observed in \cite{bk12}, this set is 
%%constructed as a union of neighborhoods 
%%of the points in $X_1 \times \D$ and $\D \times X_2$.
%%Specifically, for $z=(x_1,z_2) \in X_1 \times \D$, there is an
%%$x_2$ such that $(x_1,x_2) \in X$ and a neighborhood $N_x$ of $x$
%%(guaranteed by the edge-of-the-wedge theorem) such that each $f \in
%%\mathcal{H}(K_j)$ extends analytically to $N_x.$ Then, Proposition
%%\ref{prop:rudin2} guarantees a neighborhood $N_z$ of $z$ to which each
%%$f$ extends analytically.  It follows by the construction in the proof
%%that there is a compact set $K$ contained in $\D^2 \cup N_x$ such that
%%\[
%% \sup_{ z_0 \in N_z} |f(z_0)| \le \sup_{w \in K} |f(w)| \qquad \forall \ f \in \mathcal{H}(K_j).
%%\]
%%We can again use the uniform boundedness principle to show that  
%% the points in $\Omega_1$ are points of bounded evaluation 
%% of $\mathcal{H}(K_j)$, which finishes the proof that 
%%the points of $\Omega$ are points of bounded evaluation
%%of $\mathcal{H}(K_j).$\\

$(iii) \Rightarrow (ii)$ is immediate. \\

Now consider $(ii) \Rightarrow (i)$.  \\

First, we will show that there is a point $w\in \D^2$ where $\Phi(w)$
is invertible.  To do this, take any sequence $\{z^n\} \subset \D^2$
converging to a point $x \in X \subset \T^2$.  Since elements of
$\mcH(K_j)$ extend continuously to $X$, for each fixed $f \in
\mcH(K_j)$ the set
\[
\{\|f(z^n)\|_{E*}: n =1,2,\dots\}
\]
is bounded.  Therefore by the uniform boundedness principle for
each $j=1,2$ the set
\[
\{ \|f(z^n)\|_{E_*}: f \in \mcH(K_j), \|f\|_{\mcH(K_j)}\leq 1, n=1,2,\dots\}
\]
is bounded by say $M>0$, and this is enough to show evaluation at
$x\in X$ is bounded in $\mcH(K_j)$ and
\[
\| K_j(z^n,z^n) \|_{E_* \rightarrow E_*} \leq M^2 \text{ for each } n \text{ and } \| K_j(x,x)\|_{E_* \rightarrow E_*} 
\leq M^2
\]
for $j=1,2.$ It follows immediately that 
\begin{equation} \label{limsup}
\limsup_{n\to \infty} (1-|z_1^n|^2) K_2(z^n,z^n) = 0 \ \ \text{ and } \ \ \limsup_{n\to \infty} (1-|z_2^n|^2) K_1(z^n,z^n) = 0. 
\end{equation}
This shows that
\[
\lim_{n\to \infty} I-\Phi(z^n) \Phi(z^n)^* = \lim_{n\to \infty}
(1-|z_1^n|^2) K_2(z^n,z^n)+ (1-|z_2^n|^2) K_1(z^n,z^n) = 0
\]
and therefore for some $N \in \mathbb{N}$, $I-\Phi(z^N) \Phi(z^N)^* \leq \frac{1}{2}
I$, which implies $\Phi(z^N)$ is invertible. Set $w=z^N$. Since $\Phi$ satisfies
\[
I-\Phi(z) \Phi(w)^* = (1-z_1\bar{w}_1) K_{2,w}(z) + (1-z_2 \bar{w}_2)
K_{1,w} (z)
\]
we can extend $\Phi$ continuously to $X$ via the formula
\[
\Phi(z) = (I-(1-z_1\bar{w}_1) K_{2,w}(z) - (1-z_2 \bar{w}_2)
K_{1,w} (z))(\Phi(w)^*)^{-1}
\]
since the right hand side is assumed to be continuous.  

Finally, $\Phi$ is unitary on $X$ since for any $x\in X$, if we take a
sequence $\{z^n\}$ in $\D^2$ converging to $x$ as above, then we will
again get the result in \eqref{limsup}.  However, now that we know
$\Phi$ is continuous at $x$,
\[
0=\lim_{n\to \infty} I-\Phi(z^n) \Phi(z^n)^* = I-\Phi(x)\Phi(x)^*,
\]
which completes the proof. \end{proof}

\subsection{Canonical Realizations} \label{sect:tfr}
Unlike the previous section, we no longer assume $E,E_{*}$ are finite
dimensional. Let $\Phi \in \mathcal{S}_1(E,E_*)$ and define its de Branges-Rovnyak 
space $\mathcal{H}_{\Phi}$ to be the Hilbert space with reproducing
kernel 
\[ K_{\Phi}(z,w) := \frac{ I -\Phi(z) \Phi(w)^*} {1-z\bar{w}}.\]
Then, $\Phi$ has an (almost) unique coisometric
transfer function realization with state space equal to $\mathcal{H}_{\Phi}$
and colligation defined by
\[ U := 
\left[ \begin{array}{cc} 
A & B \\
C & D 
\end{array} \right] : 
\left[  \begin{array}{c}
\mathcal{H}_{\Phi} \\
E 
 \end{array}  \right]
\rightarrow 
\left[  \begin{array}{c}
\mathcal{H}_{\Phi} \\
E_* 
 \end{array} \right]  \]
with block operators given by
\[ \begin{aligned}
A&: f(z) \mapsto \frac{f(z) - f(0)}{z}\ \  && B:  e  \mapsto \frac{\Phi(z) - \Phi(0)}{z} e  \\ 
C &:  f(z) \mapsto f(0) \ \  &&D: e \mapsto \Phi(0)e. 
\end{aligned}
\]
Then, $ \Phi(z) = D + Cz \left( I - Az \right)^{-1}B$, and 
this representation is unique up to a minimality condition and unitary equivalence \cite{bb11}.

In two variables, transfer function realizations are more complicated
and rarely unique.
Traditionally, T.F.R.'s associated to $\Phi \in \mathcal{S}_2(E,E_*)$ are
constructed using Agler kernels $(K_1,K_2)$ of $\Phi$.
In \cite{bb11}, Ball-Bolotnikov studied T.F.R.'s defined using pairs of Agler
kernels and obtained partial characterizations of the associated block
operators $A$, $B$, $C$, and $D.$ Refined results about unitary T.F.R.'s 
for a subclass of $\mathcal{S}_d(\mathbb{D}^d)$ appear in \cite{bkvsv}; 
these are constructed in the related, but different setting of minimal 
augmented Agler decompositions.

Nevertheless, open questions about the structure of Agler kernels often
go hand in hand with open questions about the structure of T.F.R.'s. 
In this section, we use our previous analysis to clear up one such question.
Specifically, we use the concrete Agler kernels $(K^{max}_1, 
K^{min}_2)$ to construct a coisometric T.F.R. with an explicit state space  $\mathcal{M}$
and colligation $U.$ The construction answers a question 
posed by Ball and Bolotnikov in \cite{bb11}. 

\begin{rem}\label{rem:contfr}\textbf{Constructing Transfer Function Realizations.}
There is a canonical way to obtain transfer function realizations 
from Agler kernels. To illustrate this method, let $(K_1,K_2)$ be Agler
kernels of $\Phi$. Then, they satisfy
\begin{equation} \label{eqn:agform2}
 I_{E_*} - \Phi(z) \Phi(w)^* 
= (1 -z_1 \bar{w}_1) K_2(z,w) + (1-z_2\bar{w}_2) K_1(z,w). 
\end{equation}
Define the kernel functions $K_{j,w}\nu (z):= K_j(z,w) \nu$ and define
the operator $V$ by 
\[
 V: \begin{bmatrix} \bar{w}_1 K_{2,w} \nu \\
\bar{w}_2 K_{1,w} \nu \\
\nu \end{bmatrix} \mapsto 
\begin{bmatrix} K_{2,w} \nu \\
K_{1,w} \nu \\
\Phi(w)^* \nu \end{bmatrix}
\quad \forall \ w \in \mathbb{D}^2, \ \nu \in E_*.
\]
Then (\ref{eqn:agform2}) guarantees that V can be extended 
to an isometry mapping the space
\beq \mathcal{D}_V := \bigvee_{w \in \D^2, \nu \in E_*}   
 \begin{bmatrix} \bar{w}_1 K_{2,w} \nu \\
\bar{w}_2 K_{1,w} \nu \\
\nu \end{bmatrix} \subseteq \mathcal{H}(K_2) 
\oplus \mathcal{H}(K_1) \oplus E_* \eeq
onto the space
\[  \mcR_V := \bigvee_{w \in \D^2, \nu \in E_*}   
 \begin{bmatrix} K_{2,w} \nu \\
K_{1,w} \nu \\
\Phi(w)^* \nu \end{bmatrix}  \subseteq \mathcal{H}(K_2) \oplus
 \mathcal{H}(K_1) \oplus E.
 \]
Transfer function realizations with state space $\mathcal{H}(K_2) \oplus \mathcal{H}(K_1)$
are obtained by extending $V$ to a
contraction from 
\[ \mathcal{H}(K_2) \oplus \mathcal{H}(K_1) \oplus E  \rightarrow  
\mathcal{H}(K_2) \oplus \mathcal{H}(K_1) \oplus E_*\] 
and 
setting $U=V^*.$  In Ball-Bolotnikov \cite{bb11}, such a $U$ is called
a \emph{canonical functional model (c.f.m.) colligation} of $\Phi$ associated
to $(K_1, K_2).$ Similarly, 
coisometric transfer function realizations are obtained by extending 
$V$ to an isometry mapping 
\[ \mathcal{H}(K_2) \oplus \mathcal{H}(K_1) \oplus \mathcal{H} \oplus E \rightarrow
\mathcal{H}(K_2) \oplus \mathcal{H}(K_1) \oplus \mathcal{H} \oplus E_*,\] 
where $\mathcal{H}$
is an arbitrary infinite dimensional Hilbert space only added in when required, and $U$ is defined to be $V^*.$
\end{rem}

\begin{question} Let $\Phi \in \mathcal{S}(E, E_*)$.  Currently, it is
  an open question as to whether there always exists a coisometric
  transfer function realization of $\Phi$ with state space
  $\mathcal{H}(K_2) \oplus \mathcal{H}(K_1)$ for every pair of Agler
  kernels $(K_1,K_2)$.  In Section 3.2 of \cite{bb11}, Ball-Bolotnikov
  posed the following related question, which was originally stated in the
  d-variable setting:

  {\begin{center} Let $\Phi \in \mathcal{S}_2(E, E_*)$. Is there
      \emph{any} pair of Agler kernels $(K_1,K_2)$ of $\Phi$ such that
      $\Phi$ has a \emph{coisometric} c.f.m. colligation associated to $(K_1, K_2)$?
       \end{center}}
This is equivalent to asking if the construction in Remark \ref{rem:contfr} 
gives a coisometric transfer function realization of $\Phi$ with state space 
$\mathcal{H}(K_2) \oplus \mathcal{H}(K_1).$
\end{question}

The following theorem answers that question in the affirmative.

\begin{thm} \label{thm:canonicalcmf} Let $\Phi \in \mathcal{S}_2(E,E_*)$ and 
consider its Agler kernels $(K^{max}_1, K^{min}_2).$
The construction in Remark \ref{rem:contfr} gives a unique, coisometric transfer function 
realization of $\Phi$ with state space $\mathcal{H}(K^{min}_2) \oplus \mathcal{H}(K^{max}_1).$ \end{thm}

\begin{proof} Consider the construction in Remark \ref{rem:contfr}
using Agler kernels  $(K^{max}_1, K^{min}_2)$.
The operator 
$V$ is initially defined by 
\beq V: \begin{bmatrix} \bar{w}_1 K^{min}_{2,w} \nu \\
\bar{w}_2 K^{max}_{1,w} \nu \\
\nu \end{bmatrix} \mapsto 
\begin{bmatrix} K^{min}_{2,w} \nu \\
K^{max}_{1,w} \nu \\
\Phi(w)^* \nu \end{bmatrix}
\quad \forall \ w \in \mathbb{D}^2, \ \nu \in E_* \eeq
and extended to an isometry on the space
\beq \mathcal{D}_V := \bigvee_{w \in \D^2, \nu \in E_*}   
\begin{bmatrix} \bar{w}_1 K^{min}_{2,w} \nu \\
\bar{w}_2 K^{max}_{1,w} \nu \\
\nu \end{bmatrix} \subseteq \mathcal{H}(K^{min}_2) 
\oplus \mathcal{H}(K^{max}_1) \oplus E_*. \eeq
Then, transfer function realizations with state space $\mathcal{H}(K_2) \oplus
\mathcal{H}(K_1)$ are obtained by extending
$V$ to a contraction on $\mathcal{H}(K^{min}_2) \oplus \mathcal{H}(K^{max}_1) 
\oplus E_*$. We will show $\mathcal{D}_V = 
\mathcal{H}(K^{min}_2) 
\oplus \mathcal{H}(K^{max}_1) \oplus E_*.$ 
Then, the result will follow because $V$ will already be an isometry on 
$\mathcal{H}(K^{min}_2) \oplus \mathcal{H}(K^{max}_1) 
\oplus E_*$ and so we can immediately set $U=V^*$.
Define
\[
\mathcal{D}  := \bigvee_{w \in \D^2, \nu \in E_*}   
 \begin{bmatrix} \bar{w}_1 K^{min}_{2,w} \nu \\
\bar{w}_2 K^{max}_{1,w}\nu  \end{bmatrix} 
\subseteq \mathcal{H}(K^{min}_2) \oplus \mathcal{H}(K^{max}_1). 
\]
Examining the case $w=0$ shows that $\mathcal{D}_V$ coincides with
$\mathcal{D} \oplus E_*$, so it suffices to show $\mathcal{D}=
\mathcal{H}(K^{min}_2) \oplus \mathcal{H}(K^{max}_1).$ 
Assume 
\beq \begin{bmatrix} f_2 \\ f_1 
\end{bmatrix} \in \left[ \mathcal{H}(K^{min}_2) 
\oplus \mathcal{H}(K^{max}_1) \right] \ominus \mathcal{D}. \eeq
Then for each $w \in \mathbb{D}^2$ and $\nu \in E_*$, 
\begin{align*} 0 &= \LL \begin{bmatrix} f_2 \\ f_1 
\end{bmatrix}, \begin{bmatrix} \bar{w}_1 K^{min}_{2,w} \nu \\
\bar{w}_2 K^{max}_{1,w}\nu  \end{bmatrix} \RR_{\mathcal{H}(K^{min}_2) 
\oplus \mathcal{H}(K^{max}_1)} \\
&& \\
&=  w_1 \LL f_2,  K^{min}_{2,w}\nu \RR_{\mathcal{H}(K^{min}_2)}+
 w_2 \LL f_1,  K^{max}_{1,w}\nu \RR_{\mathcal{H}(K^{max}_1)} \\
 &\\
& = \LL w_1f_2(w) + w_2f_1(w), \nu \RR_{E_*},
\end{align*}
which implies $Z_1f_2 + Z_2f_1 = 0.$ Thus, there is some $F \in H^2(E_*)$
such that $f_1 =Z_1 F.$ Now, since $f_1 \in \mathcal{H}(K^{max}_1)$, there
is a $g_1 \in Z_1 L^2_{--}(E)$ such that
\begin{equation} \label{eqn:maxcon} \begin{bmatrix}
f_1 \\
g_1 
\end{bmatrix} \in \mcR_1
\ominus Z_1 \mcR.\end{equation}
This also gives $g_1 - \Phi^* f_1 \in \Delta L^2(E)$ and  a 
$G \in L^2_{--}(E)$ with $g = Z_1 G.$  Since $\Delta L^2(E)$ 
is invariant under $Z^*_1$, is is clear that 
$G - \Phi^* F\in \Delta L^2(E)$ as well. Then
\beq \begin{bmatrix}
f_1 \\
g_1 
\end{bmatrix} 
= Z_1  \begin{bmatrix}
F \\
G 
\end{bmatrix}
\text{ and }  \begin{bmatrix}
F \\
G 
\end{bmatrix} \in \mcR. \eeq
Given this, (\ref{eqn:maxcon}) forces $f_1 \equiv 0$, so
 $f_2 \equiv 0$ and $\mathcal{D} 
= \mathcal{H}(K^{min}_2) \oplus  \mathcal{H}(K^{max}_1).$
\end{proof}

\begin{rem}\textbf{The Canonical Block Operators.} Let 
$U$ be the operator associated to the transfer function 
realization given in Theorem \ref{thm:canonicalcmf}. 
Much can be said about its block operators $A,B,C,D$. 
In the setting of general $(K_1,K_2)$, much of this analysis 
already appears in \cite{bb10} and \cite{bb11}. We will first give the formulas
for $A,B,C,D$ and then discuss the derivations.
 Specifically, for every $ f := \begin{bmatrix}
f_1 \\ 
f_2
\end{bmatrix}  \in \mathcal{H}(K^{min}_2) \oplus \mathcal{H}(K^{max}_1)$
and $\eta \in E$, 
\[ C: 
\begin{bmatrix}
f_1 \\ 
f_2
\end{bmatrix}
\mapsto 
f_1(0) + f_2(0) 
\ \text{ and } \
D: \eta \mapsto \Phi(0) \eta.\]
For $A$ and $B$, let us first simplify notation by setting
\[ 
\begin{bmatrix}
(Af)_1 \\
(Af)_2 
\end{bmatrix} 
:=A
 \begin{bmatrix}
f_1 \\
f_2
\end{bmatrix}
\ \text{ and } 
\begin{bmatrix}
(B \eta)_1 \\
(B \eta)_2
\end{bmatrix} 
:= B \eta. \]
Then $(Af)_2$ and $(B \eta)_2$ are the unique functions in $\mathcal{H}(K^{max}_1)$ satisfying
\begin{align*}
 \left(Af \right)_2 (0, w_2) &= \frac{ f_1(0, w_2) - f_1(0) + f_2(0,w_2)-f_2(0)}{w_2} \\
\left( B \eta \right)_2(0,w_2) & = \frac{\Phi(0,w_2)- \Phi(0)}{w_2} \eta,
\end{align*}
for all $w_2 \in \mathbb{D} \setminus \{0\},$ and $(Af)_1$ and $(B \eta)_1$ are the unique functions in $\mathcal{H}(K^{min}_2)$ satisfying
\begin{align*}  \left( Af \right)_1(w) &= \frac{f_1(w) - f_1(0) + f_2(w)-f_2(0) -w_2 \left(Af \right)_2 (w)}{w_1} \\
 \left( B \eta \right)_1(w) & = \frac{\left( \Phi(w)- \Phi(0) \right)\eta - w_2 \left( B \eta \right)_2(w)}{w_1},
\end{align*}
for all $w \in \mathbb{D}^2$ with $w_1 \ne 0.$ The results for $C$ and $D$ follow because, by definition
\[ U^* = 
\begin{bmatrix}
A^* & C^* \\
B^* & D^* 
\end{bmatrix}
: \begin{bmatrix} \bar{w}_1 K_{2,w}^{min} \nu \\
\bar{w}_2 K_{1,w}^{max} \nu \\
\nu \end{bmatrix} \mapsto 
\begin{bmatrix}  K_{2,w}^{min} \nu \\
K_{1,w}^{max} \nu \\
\Phi(w)^* \nu \end{bmatrix}
\quad \forall \ w \in \mathbb{D}^2, \ \nu \in E_*.
\]
Setting $w=0$ immediately implies that
\[ C^*: \nu \mapsto 
\begin{bmatrix}
 K^{min}_{2,0}\nu \\
 K^{max}_{1,0} \nu 
  \end{bmatrix}
\text{ and }
D^*: \nu \mapsto \Phi(0)^*\nu
\]
for all $\nu \in E_*$. Then the calculations
\[ \LL C 
\begin{bmatrix} f_1 \\
f_2
\end{bmatrix},
\nu
\RR_{E_*} 
=\LL 
\begin{bmatrix} f_1 \\
f_2
\end{bmatrix},
\begin{bmatrix}
 K^{min}_{2,0} \nu \\
 K^{max}_{1,0}\nu
  \end{bmatrix} \RR_{\mathcal{H}(K^{min}_2) \oplus \mathcal{H}(K^{max}_1)} 
  =
  \LL f_1(0) + f_2(0), \nu \RR_{E_*}
   \]
and 
\[ \LL D \eta, \nu \RR_{E_*}
= \LL \eta, D^* \nu \RR_{E}
=\LL \eta, \Phi(0)^* \nu \RR_{E}
=\LL \Phi(0) \eta, \nu \RR_{E_*}
\]  
give the formulas for $C$ and $D$. Moreover,
The results about  $C^*$ and $D^*$ imply that
\[ A^*: 
\begin{bmatrix}
\bar{w}_1 K_{2,w}^{min} \nu \\
\bar{w}_2 K_{1,w}^{max} \nu
\end{bmatrix}
\mapsto
\begin{bmatrix}
\left( K_{2,w}^{min} - K^{min}_{2,0} \right) \nu \\
\left( K_{1,w}^{max} - K^{max}_{1,0} \right) \nu
\end{bmatrix}\]
 and 
\[
B^*: 
\begin{bmatrix}
\bar{w}_1 K_{2,w}^{min} \nu \\
\bar{w}_2 K_{1,w}^{max} \nu
\end{bmatrix}
\mapsto
\left( \Phi(w)^* - \Phi(0)^* \right) \nu.\]
Then 
\[ 
\begin{aligned}
\LL w_1(Af)_1 (w) + w_2(Af)_2 (w), \nu \RR_{E_*} &= \LL Af, \begin{bmatrix}
\bar{w}_1 K_{2,w}^{min} \nu \\
\bar{w}_2 K_{1,w}^{max} \nu
\end{bmatrix}\RR_{\mathcal{H}(K^{min}_2) \oplus \mathcal{H}(K^{max}_1)}\\
&\\
& = \LL \begin{bmatrix}
f_1 \\
f_2
\end{bmatrix}, 
\begin{bmatrix}
\left( K_{2,w}^{min} - K^{min}_{2,0} \right) \nu \\
\left( K_{1,w}^{max} - K^{max}_{1,0} \right) \nu
\end{bmatrix} \RR_{\mathcal{H}(K^{min}_2) \oplus \mathcal{H}(K^{max}_1)}
 \\
&\\
& = \LL f_1(w) - f_1(0) + f_2(w) - f_2(0), \nu \RR_{E_*},
\end{aligned}
\]
and similarly,
\[ 
\LL w_1(B \eta)_1 (w) + w_2(B \eta)_2 (w), \nu \RR_{E_*} 
= \LL \left( \Phi(w) - \Phi(0) \right)\eta, \nu \RR_{E_*}.
\]
Therefore, we have
\begin{align}
\label{eqn:gleason1}  w_1 \left( Af \right)_1(w) + w_2 \left(Af \right)_2 (w) &= f_1(w) - f_1(0) + f_2(w)-f_2(0) \\
\label{eqn:gleason2}  w_1 \left( B \eta \right)_1(w) + w_2 \left( B \eta \right)_2(w) & = \left( \Phi(w)- \Phi(0) \right)\eta.
\end{align}
Operators that solve $(\ref{eqn:gleason1})$ or $(\ref{eqn:gleason2})$ are said to solve the structured Gleason problem
for $\mathcal{H}(K^{min}_2) \oplus \mathcal{H}(K^{max}_1)$  or for $\Phi$, respectively. In general, such operators
are not unique. However, in this situation, $A$ and $B$ are uniquely determined. The proof of this rests on two observations.
First, when $w_1=0$ and $w_2 \ne 0,$ $(\ref{eqn:gleason1})$ and $(\ref{eqn:gleason2})$ become
\begin{align}
\label{eqn:gleason3}   \left(Af \right)_2 (0, w_2) &= \frac{ f_1(0, w_2) - f_1(0) + f_2(0,w_2)-f_2(0)}{w_2} \\
\label{eqn:gleason4}   \left( B\eta \right)_2(0,w_2) & = \frac{ \Phi(0,w_2)- \Phi(0)}{w_2}\eta.
\end{align}
It is also true that the set $\{ (0,w_2) : w_2 \in \mathbb{D} \setminus \{0\}\}$ is a set of uniqueness for $\mathcal{H}(K^{max}_1).$ Indeed,
suppose two functions $g_1, g_2 \in \mathcal{H}(K^{max}_1)$ satisfy $g_1(0,w_2) = g_2(0,w_2)$ for all $w_2 \ne 0$. This immediately
implies $g_1(0,0)=g_2(0,0)$ and
\[ g_1- g_2 =Z_1 h \] for some $h \in H^2(E_*).$ Arguments identical
to those in the proof of Theorem \ref{thm:canonicalcmf} show that $h$
must be zero, so $g_1=g_2.$  As $\left(Af \right)_2$ and $\left( B \eta \right)_2$
are in $\mathcal{H}(K^{max}_1)$, they must be the unique such
functions satisfying $(\ref{eqn:gleason3})$ and $(\ref{eqn:gleason4})$
respectively. Then, the other components $\left(Af \right)_1$ and
$\left( B \eta \right)_1$ are uniquely determined by
$(\ref{eqn:gleason1})$ and $(\ref{eqn:gleason2}).$ In one-variable,
$Af$ and $B \eta$ can be explicitly written in terms of $f$ and
$\eta.$ Given that, our characterizations of $A$ and $B$ seem slightly
unsatisfying. This motivates the question
\begin{question} Assume $g \in \mathcal{H}(K^{max}_1)$. Is there an explicit way to construct $g$ using only the function $g(0,w_2)?$ 
\end{question}
A clean answer would also provide nice formulas for the operators $A$ and $B$. 
It seems possible that the refined results in \cite{bkvsv} about unitary T.F.R.'s associated to minimal 
augmented Agler decompositions might suggest methods of answering this question.
\end{rem}

\section{Appendix: Vector valued RKHS's} \label{sect:opkernels}

In this section, we record several facts about vector valued reproducing 
kernel Hilbert spaces that were used in earlier sections. The results are 
well-known in the scalar valued case. See, for example \cite{aro50}, 
\cite{bv03b}, Chapter 2 in \cite{alp01}, and Chapter 2 in \cite{ ampi}. 
We outline how the needed vector valued results follow from the known 
scalar valued results. Let $\Omega$ be a set and $E$ be a separable Hilbert
space. We will frequently use the following observation:

\begin{rem} For each function $f: \Omega \rightarrow E$ there is an 
associated scalar valued function $\tilde{f}: \Omega \times E 
\rightarrow \mathbb{C}$ defined as follows:
\[
 \tilde{f}(z, \eta) := \LL f(z), \eta \RR_{E}. 
\]
If functions $f,g: \Omega \rightarrow E$ and $\tilde{f} \equiv \tilde{g}$, 
then $f \equiv g.$
\end{rem}

\begin{defn} \label{defn:scalarhs} Let $\mathcal{H}(K)$ be a reproducing 
kernel Hilbert space of $E$ valued functions on $\Omega$.
For $w\in \Omega$ and  $\nu \in E$, define the function $K_w \nu:= K( \cdot, w)\nu.$ An associated
reproducing kernel Hilbert space of scalar valued functions on $\Omega \times E$
 can be defined as follows: Define the set of functions
\[ \mathcal{H} := \left \{ \tilde{f}:  f \in \mathcal{H}(K) \right \} 
\]
and equip $\mathcal{H}$ with the inner product 
\[
 \LL \tilde{f}, \tilde{g} \RR_{\mathcal{H}} 
= \LL f ,g \RR_{\mathcal{H}(K)}. 
\]
It is routine to show that $\mathcal{H}$ is a Hilbert space 
with this inner product and since
\[
 \tilde{f}(w,\nu) =  \LL f(w), \nu \RR_{E} 
= \LL f, K_w \nu \RR_{\mathcal{H}(K)} 
= \LL \tilde{f}, \widetilde{ K_w \nu} \RR_{\mathcal{H}},
\]
$\mathcal{H}$ is a reproducing kernel Hilbert space with reproducing kernel
\[
L \big( (z,\eta), (w,\nu) \big) := \widetilde{K_w\nu} ( z, \eta) 
= \LL K(z,w) \nu, \eta \RR_{E} = \eta^* K(z,w) \nu.
\]
Then $f \in \mathcal{H}(K)$ if and only if $\tilde{f} \in \mathcal{H}(L)$.
 It is also clear that $\|f \|_{\mathcal{H}(K)} = \| \tilde{f}\|_{\mathcal{H}(L)}.$ 
 \end{defn}

The following results are well-known for scalar valued reproducing kernel 
Hilbert spaces and follow easily for vector valued reproducing kernel Hilbert spaces.

\begin{thm} \label{thm:kerdiff} Let $\mathcal{H}(K)$ and $\mathcal{H}(K_1)$ 
be reproducing kernel Hilbert spaces of $E$ valued functions on $\Omega$. 
Then $\mathcal{H}(K_1) \subseteq \mathcal{H}(K)$ contractively if and only if 
\[
K(z,w) - K_1(z,w) \text{ is a positive kernel.} \] \end{thm}

\begin{proof} As in Definition \ref{defn:scalarhs}, consider the Hilbert spaces 
$\mathcal{H}(L)$ and $\mathcal{H}(L_1)$ of scalar valued functions on 
$\Omega \times E$ with reproducing kernels given by
\[
  L \big( (z,\eta), (w,\nu) \big) :=\eta^* K(z,w)  \nu \ \ \text{ and }  
L_1 \big( (z,\eta), (w,\nu) \big) :=\eta^* K_1(z,w)  \nu.
\]
It is routine to show that $\mathcal{H}(K_1) \subseteq \mathcal{H}(K)$
contractively if and only if $\mathcal{H}(L_1) \subseteq
\mathcal{H}(L)$ contractively. It follows from well-known scalar
results, which appear on page 354 of \cite{aro50}, that
$\mathcal{H}(L_1) \subseteq \mathcal{H}(L)$ contractively if and only
if
\[
 L(z,w) - L_1(z,w) \text{ is a positive kernel.} 
\]
The result follow from the fact that $L(z,w) - L_1(z,w)$ is a 
positive kernel if and only if $K(z,w)-K_1(z,w)$ is a positive kernel.
\end{proof}

Similarly, the following two results can be deduced from the 
scalar-valued case:

\begin{thm} \label{thm:kermult} Let $\mathcal{H}(K)$ be a 
reproducing kernel Hilbert space of $E$ valued functions on 
$\Omega$ and let $\psi: \Omega \rightarrow \mathbb{C}$. 
Then $\psi$ is a multiplier of $\mathcal{H}(K)$ with multiplier 
norm bounded by one if and only if
\[
 \big( 1 - \psi(z) \overline{ \psi(w)} \big) K(z,w) \text{ is a positive kernel}. 
 \]
\end{thm}

\begin{proof} When we say ``$\psi$ is a multiplier of $\mathcal{H}(K)$,"
we mean that $\psi \otimes I_{\mathcal{H}(K)}$ maps $\mathcal{H}(K)$ into $\mathcal{H}(K).$

Now, using the definition of $\mathcal{H}(L)$, it is easy to 
show that $\psi$ is a multiplier of $\mathcal{H}(K)$ with multiplier 
norm bounded by one if and only if $\psi$ is a multiplier of 
$\mathcal{H}(L)$ with multiplier norm bounded by one. By the 
analogous scalar valued result, which appears as Corollary 
2.3.7 in \cite{ampi}, 
it follows that $\psi$ is a multiplier of $ \mathcal{H}(L)$ with multiplier 
norm bounded by one if and only if
\[
 \big( 1 - \psi(z) \overline{ \psi(w)} \big) L \big( (z, \eta), 
(w, \nu) \big) \text{ is a positive kernel}. 
\]
The result then follows by using the definition of a positive kernel 
to show that  $\big( 1 - \psi(z) \overline{ \psi(w)} \big) L \big( (z, \eta),
 (w, \nu) \big)$ is a positive kernel if and only if 
$\big( 1 - \psi(z) \overline{ \psi(w)} \big) K(z,w) $ is a positive kernel.
 \end{proof}

\begin{thm}  \label{thm:kersum} Let $\mathcal{H}(K_1), 
\mathcal{H}(K_2)$ be reproducing kernel Hilbert spaces
 of $E$ valued functions on $\Omega$. Then $\mathcal{H}(K_1 + 
K_2)$ is precisely the Hilbert space composed of the set of functions
\[
 \mathcal{H}(K_1) + \mathcal{H}(K_2) := \left \{ f_1 +f_2 : 
f_j \in \mathcal{H}(K_j) \right \}. 
\]
equipped with the norm
\[
 \| f \|^2_{\mathcal{H}(K_1+K_2)} = \min_{\substack{ f = f_1 + f_2 \\ f_j
 \in \mathcal{H}(K_j)}} \|f_1\|^2_{\mathcal{H}(K_1)} 
 + \|f_2 \|^2_{\mathcal{H}(K_2)}. 
 \]
\end{thm}

\begin {proof} As before consider the related scalar valued reproducing
 kernel Hilbert spaces $\mathcal{H}(L_1)$ and $\mathcal{H}(L_2)$, where 
 \[
  L_1 \big( (z,\eta), (w,\nu) \big) :=\eta^* K_1(z,w)  \nu \ \ \text{ and } 
 L_2 \big( (z,\eta), (w,\nu) \big) :=\eta^* K_2(z,w)  \nu. 
\]
The analogous scalar valued result, which appears on page 353 in \cite{aro50}, states
 $\mathcal{H}(L_1 + L_2)$ is precisely the Hilbert space composed of the
set of functions
\[
\mathcal{H}(L_1) + \mathcal{H}(L_2) := \left \{ f_1 +f_2 : f_j \in
\mathcal{H}(L_j) \right \}. 
\]
equipped with the norm
\[
 \| f \|^2_{\mathcal{H}(L_1+L_2)} = \min_{\substack{ f = f_1 + f_2
 \\ f_j \in \mathcal{H}(L_j)}} \|f_1\|^2_{\mathcal{H}(L_1)} + 
 \|f_2 \|^2_{\mathcal{H}(L_2)}, 
 \]
Using this and the connections between $\mathcal{H}(L_j)$ 
and $\mathcal{H}(K_j)$,  it is easy to deduce the desired result. 
The details are left as an exercise.\end{proof}

%% \begin{question} Can all Agler kernels be characterized in 
%% terms of $K^{max}_j$ and $K^{min}_j$? \end{question}

\end{document}